\documentclass[a4paper]{amsart}
\usepackage{MyCommands}
\usepackage[style = alphabetic, giveninits]{biblatex}
\usepackage{hyperref}

\bibliography{references.bib}

\title{Contact geometry in infinite dimensions}
\author[F.A.K. Sanders]{\orcidlink{0009-0007-8603-8753}
Fraser A.K. Sanders}

\address{Mr F.A.K. Sanders, The Alan Turing Building, The University of Manchester, Manchester, M13 9PL}
\email{fraser.sanders@manchester.ac.uk}
\keywords{Contact geometry, symplectic geometry, Hamiltonian mechanics, Banach spaces, Fréchet spaces, infinite-dimensional geometry}
\subjclass[2020]{53D10, 37J55, 70H05, 46T05}

\begin{document}

\begin{abstract}
    We generalise the theories of cosymplectic, contact, and cocontact manifolds to the infinite-dimensional setting and calculate model examples of time-dependent and dissipative Hamiltonian systems.
\end{abstract}

\maketitle
\tableofcontents
\section{Introduction}
Symplectic geometry and the theory of Hamiltonian mechanics can be readily extended to infinite dimensions, with some very useful results. See \cite{MarsdenWeinstein83, ArnoldKhesin21} for only a few examples. Since these Hamiltonian systems can only model physical systems where the Hamiltonian is conserved, it is natural to want to extend this theory to physical systems in which the Hamiltonian, and other properties, are instead dissipated, or diffused. 

In finite dimensions, contact geometry is a well-studied companion to symplectic geometry, often for the fact that one can define Hamiltonian vector fields which take into account some notion of dissipation or diffusion, particularly in the case of thermodynamics (see \cite{Bravetti17} for a survey on this, and \cite{Mrugala78} for a seminal paper in the field of geometrothermodynamics). Hence, it is the perfect place to start.

However, the definition of a contact form makes explicit use of the dimension of the manifold: on a $2n+1$ dimensional manifold, $\eta$ is a contact form if $\eta \wedge (d\eta)^n$ is a volume form. This definition is unsuitable if we wish to extend the concept of contact Hamiltonian mechanics into infinite dimensions. The aim of this paper is to redefine contact geometry, without making explicit the dimension of the underlying manifold, and then to apply this new definition to a range of infinite-dimensional manifolds and contact Hamiltonian systems. Physically, these correspond to systems that must necessarily be modelled by infinite-dimensional spaces, and have some notion of irreversibility, diffusion, or dissipation.

This paper can be summarised by the following ``commutative diagram''.
\begin{figure*}[h!]
    \centering
    \begin{tikzcd}
        \textrm{Symplectic Geometry} \arrow[r, "\infty"] \arrow[dd, "\textrm{non-conservative}"'] & \infty\textrm{-dim Symplectic Geometry} \arrow[dd, "\textrm{non-conservative}"] \\ \\
        \textrm{Contact Geometry} \arrow[r, "\infty"] & \infty\textrm{-dim Contact Geometry}
    \end{tikzcd}
\end{figure*}

The outline of the paper is as follows: this paper starts with a background section in section 2. It can be freely skipped by any reader with knowledge of geometry in infinite dimensions, or referred back to to fix notation. Section 3 then defines (or redefines) contact distributions in infinite dimensions, and explores which properties are preserved from the finite-dimensional case, and which properties arise in the new setting. In section 4, we build examples of infinite-dimensional contact manifolds, and prove that some constructions from finite dimensions still work in the infinite-dimensional setting, in particular
contactisations and hypersurfaces of contact type. Section 5 starts to generalise the theory of contact Hamiltonian mechanics to the infinite-dimensional setting. This section contains two main examples. The first is a toy example on the space of 1-forms on a compact, orientable, Riemannian 2-manifold. This example is of limited physical interest, but may be a useful example for studying the action of diffeomorphism groups. The second example is the damped harmonic oscillator, whose physical interest should hopefully be clear. The final section, section 6, takes the work done in the previous sections, and extends it further to define cosymplectic and cocontact manifolds. We then use these structures to define Hamiltonians for a wave equation with a source term, and a damped wave equation with a source term, respectively.

\subsection*{Acknowledgements}
The author would like to express his gratitude to Prof. Kasia Rejzner for her encouragement at the very beginning of this project, especially for pointing me in the direction of \cite{KrieglMichor97}. The author is also grateful to the rest of the Mathematical Physics and Quantum Information research group at the University of York, whose helpful questions at their seminar guided much of section 4 of this paper.

The author would like to thank Dr Robert Cardona for his helpful discussion over lunches at the XXXIInd International Fall Workshop on Geometry and Physics, as well as the organising committee of the XXXIIIrd edition of the workshop for the opportunity to present some of the following results. 

Thanks go to Jiacheng Tang and Matthew Antrobus for their help in understanding the subtleties of topological vector spaces, as well as Andreas Swerdlow for discussions on differential geometry.

Finally, the author would like to thank Dr James Montaldi for many hours of fruitful discussion, and for bringing to his attention the exact combination of papers that inspired the idea central to this paper.
\section{Background}
The following is a very brief introduction to differential geometry in infinite dimensions, specifically symplectic geometry in infinite dimensions.

\subsection{Topological vector spaces}
In some sense the most primitive setting for differential geometry in infinite dimensions is a \emph{topological vector space}. All vector spaces in this paper will be over $\R$, which we equip with the Euclidean topology.

\begin{definition}[TVS and lcTVS]
    A \emph{topological vector space} (TVS) is a vector space $E$, equipped with a topology such that the maps 
    \begin{align*}
        +\colon E \times E \rightarrow E, && \times \colon \R \times E \rightarrow E
    \end{align*}
    are continuous.

    A TVS is futher \emph{locally convex} (an lcTVS) if 0 has a local basis consisting of convex sets.
    
    The word \emph{toplinear} is sometimes used to describe a map between two TVS's that is both continuous and linear. Such maps are the morphisms in the category of TVS's (or any full subcategory thereof). 
\end{definition}
We will assume all TVS's to be Hausdorff.

The need for topological vector spaces for differential geometry can be summed up briefly by considering expressions such as
\[\lim_{\delta \rightarrow 0} \frac{f(x+\delta)-f(x)}{\delta},\]
which requires both a topology and a vector space structure to make sense.
However, as always in topology, for many purposes we must make extra assumptions on the topology to avoid pathological behaviour. In particular, the chain rule does not hold in  locally convex spaces generally.

We will denote the \emph{algebraic dual}, that is, the space of linear maps from a TVS $E$, into $\R$ as $E^*$. We will denote the space of $continous$ linear maps $E\rightarrow \R$ as $E'$.

For the purposes of this paper we concentrate on \emph{convenient spaces}, \emph{Fréchet spaces}, and \emph{Banach spaces}.
\subsubsection{Smooth curves and convenient spaces}
We paraphrase \cite[\textsection 1-2]{KrieglMichor97}.
\begin{definition}[smooth maps]
    Let $E$ be a lcTVS and $c \colon \R \rightarrow E$ be a continuous (not necessarily linear) function. Then we call $c$ a $C^0$ path. We call $c$ \emph{differentiable} or $C^1$ if the \emph{derivative}
    \[ c'(t):=\lim_{\delta\rightarrow 0} \frac{c(t+\delta)-c(t)}{\delta} \]
    exists and is a continuous function of $t$. Further, $c$ is $C^n$ if the $n$-th iterated derivative exists and is continuous, and $\Cinf$ if it is $C^n$ for all $n$. We write $\Cinf(\R, E)$ for the space of smooth paths in $E$.

    Let $E$ and $F$ be convenient vector spaces and $U\subset E$ be open. Then a map $f\colon U\rightarrow F$ is called a \emph{smooth map} if $f\circ c$ is a smooth curve in $F$ for all smooth curves $c\in \Cinf(\R, E)$.
\end{definition}

The following definition is only one of 7 equivalent conditions given by the authors in \cite{KrieglMichor97} but it is the most intuitive for differential geometry.
\begin{definition}[Convenient vector space]
    An lcTVS, $E$,  is \emph{convenient} (sometimes called \emph{Mackey-complete} or \emph{$c^\infty$-complete}) if for any $c_1 \in \Cinf(\R, E)$ there exists some $c_2 \in \Cinf(\R,E)$ such that $c_2' = c_1$.

    Equivalently, if $c \in C^0(\R, E)$ and for all $\ell \in E'$, $\ell \circ c\colon \R \rightarrow \R$ is smooth, then $c$ is smooth.
\end{definition}

Importantly, the strong dual of a convenient vector space is also a convenient vector space. In fact the space $\Cinf(E,F)$ is a convenient vector space for all pairs, $E,F$, of convenient vector spaces.

\subsubsection{Fréchet spaces}
\begin{definition}[seminorm]
    A \emph{seminorm} on a vector space $E$ is a function $p\colon E \rightarrow \R$ such that 
    \begin{itemize}
        \item $p(x)\geq 0$ for all $x \in E$,
        \item $p(x+y) \leq p(x)+ p(y)$ for all $x,y \in E$,
        \item $p(\lambda x) = |\lambda|p(x)$ for all $\lambda \in \R$ and $x \in E$.
    \end{itemize}
    Note that a seminorm is only one assumption away from a norm --- non-zero vectors may have seminorm zero.
\end{definition}

Following \cite{Dodson15}, given a family of seminorms, $\Gamma = \braces{p_i}_{i \in I}$, on a vector space $E$, we can fix a topology $\tau_\Gamma$ determined by the family 
\[ B_\Gamma := \{S(\Delta, \varepsilon) | \ \varepsilon>0, \ \Delta \subset \Gamma, \text{$\Delta$ finite}\} \]
where 
\[ S(\Delta, \varepsilon) := \braces{x \in E | \ p(x)<\varepsilon, \forall p \in \Delta } \]

Then we have the following statements \cite[prop. 2.1.3]{Dodson15}:
\begin{itemize}
    \item $\tau_\Gamma$ makes $E$ into a topological vector space, with the finest topology making all seminorms continuous.
    \item A TVS is locally convex if and only if the topology coincides with $\tau_\Gamma$ for some system of seminorms.
    \item $(E, \tau_\Gamma)$ is Hausdorff iff 
        \[ x = 0 \Leftrightarrow \left(p(x) = 0, \, \forall p \in \Gamma\right). \]
    \item If $(E, \tau_\Gamma)$ is Hausdorff, then it is also metrisable exactly whenever $\Gamma$ is countable.
    \item Let $(x_n)_{n\in \mathbb{N}}$ be a sequence in $(E, \tau_\Gamma)$. Then
    \[ (x_n \to x) \Leftrightarrow (\forall p \in \Gamma, \, p(x_n - x) \rightarrow 0). \]
    \item We call a sequence $(x_n)_{n\in \mathbb{N}}$ \emph{Cauchy} if 
    \[\lim_{n.m \to \infty} p(x_n-x_m) \to 0, \ \forall p \in \Gamma \]
    and we call $(E, \tau_\Gamma)$ \emph{complete} if all Cauchy sequences converge.
\end{itemize}

\begin{definition}[Fréchet space]
    A \emph{Fréchet space} is a locally convex, Hausdorff, metrisable, and complete topological vector space.
\end{definition}
For this paper the most important Fréchet spaces are spaces of sections of some given vector bundle, $\Gamma(V \rightarrow M)$, where $M$ is a compact finite-dimensional Riemannian manifold. This includes $\Cinf(M)$, $\vect(M)$, and $\Omega^k(M)$. The seminorms in these examples are given by 
\[\|f\|_n = \sum_{i=0}^n \sup_{x\in M} |\nabla^i f|\]
for some choice of covariant derivative, $\nabla$.

Unlike convenient vector spaces, the dual space of a Fréchet space is not necessarily a Fréchet space. In fact, for a Fréchet space $E$, $E'$ is Fréchet iff $E$ is Banach.

\subsubsection{Banach spaces}
\begin{definition}[Banach space]
    Let $E$ be a vector space and $\|.\|\colon E \rightarrow \R$ be a norm (that is, a seminorm such that $\|x\| = 0 \Leftrightarrow x=0$). Then we can induce a metric, and hence a topology, on $E$ by defining $d(x,y):= \|x-y\|$. If this metric is complete, then $(E, \|.\|)$ is a \emph{Banach space.}
\end{definition}

The dual of a Banach space is once again a Banach space, as well as the space $L(E,F)$ of continuous linear maps between Banach spaces $E$ and $F$. If a space $E$ is isomorphic to its double-dual $E''$, then we call $E$ \emph{reflexive}, and if $E\cong E'$ then $E$ is \emph{self-dual}. One important class of self-dual Banach spaces is Hilbert spaces.

\subsection{Differential geometry in infinite dimensions}
\begin{definition}[Manifold]
    Fix a TVS $E$ and a topological space $M$. Let $U\subset M$ be open. A \emph{chart} $(\varphi, U)$ is a homeomorphism $\varphi: U \rightarrow V \subset E$, for an open set $V$. Two charts $(\varphi_\alpha, U_\alpha)$, $(\varphi_\beta, U_\beta)$ are \emph{compatible} if the map $\varphi_\alpha \circ \varphi_\beta^{-1}$ is smooth where it is defined.

    A family $(\varphi_\alpha, U_\alpha)_{\alpha \in A}$ of charts is called an \emph{atlas} if $\braces{U_\alpha}_{\alpha \in A}$ is a cover for $M$ and each pair of charts is compatible. As in finite-dimensional manifolds, we may take \emph{maximal atlases} and then we call the pair $(M, \mathcal{A})$ consisting of a topological space and a maximal atlas a \emph{manifold modelled by $E$}.
\end{definition}

Let $M$ be modelled by a TVS $E$. If $E=\R^n$ for some finite $n$ then $M$ is a \emph{finite dimensional manifold} and the definition coincides with the more familiar one. If $E$ is convenient then $M$ is a \emph{convenient manifold}, and so on. As previously mentioned \cite{KrieglMichor97} is a comprehensive source on the geometry of convenient manifolds and all results hold for Fréchet and Banach manifolds, but \cite{Dodson15} and \cite{Lang72} are good sources for the geometry of Fréchet manifolds and Banach manifolds respectively. The former specifically uses the approach of identifying each Fréchet space with an inverse limit in the category of Banach spaces.

\begin{definition}
    Let $M$, $N$ be manifolds modelled by $E$, a TVS. A continuous path $c\colon \R \rightarrow M$ is \emph{smooth} if for any chart $\varphi$, $\varphi \circ c$ is a smooth path in $E$, wherever it is defined. The space of smooth paths on a manifold $M$ is denoted $\Cinf(\R, M)$.
    
    Let $F\colon M\rightarrow N$ be a continuous map. $F$ is \emph{smooth} if for all smooth paths $c \in \Cinf(\R, M)$, $F\circ c$ is a smooth path in $N$.
\end{definition}

\begin{definition}[Kinematic tangent bundle]
    Let $c_1, c_2 \in \Cinf(\R, M)$. Define an equivalence relation $\sim$ by
    \[c_1 \sim c_2 \Leftrightarrow c_1(0) = c_2(0), \ \text{and} \ (\varphi \circ c_1)'(0) = (\varphi \circ c_2)'(0).\]

    Define the set $TM = \Cinf(\R, M)/\sim$, and the map $\pi\colon TM \rightarrow M$ by $\pi([c]) = c(0)$. By \emph{consistent transfer of linear structure} \cite[prop. 2.3]{JMLee09}, the sets $\pi^{-1}(x)=: T_xM$ are each isomorphic to $E$ and $\pi\colon TM \rightarrow M$ is a vector bundle over $M$. In particular, $TM$ is a manifold modelled by $E\times E$ which we call the \emph{kinematic tangent bundle}. Refer to \cite[\textsection 28]{KrieglMichor97} for details on the local charts. If $v \in T_xM$ then we call $x$ the \emph{foot-point} of $v$.
\end{definition}

With the above definitions, tangent maps and the tangent functor behave as in finite dimensions.

\begin{definition}[Kinematic cotangent bundle]
    We can also construct the kinematic cotangent bundle in the same manner as above, but applying the dual functor to the model space $E'$, so that $T'M$ is modelled by $E\times E'$.
\end{definition}

Note that although for a Fréchet space $E$, $E'$ is no longer Fréchet, it is still convenient, so as long as we consider the larger category of convenient manifolds, it still makes sense to consider the cotangent bundle of a Fréchet manifold.

\begin{definition}[Kinematic vector fields and 1-forms]
    A \emph{kinematic vector field} is a section of $TM$. The set of kinematic vector fields is denoted $\vect(M)$. For this paper we will often consider $\vect(M)$ as a module over $\Cinf(M):= \Cinf(M,\R)$.

    Likewise, a \emph{kinematic 1-form} is a section of $T'M$ . We denote this space $\Omega^1(M)$.
\end{definition}

Note that the word ``kinematic'' has appeared in our definitions. This is because, unlike in Banach manifolds, the space of sections of $TM$ is not the same as the space of derivations of $\Cinf(M)$, the latter are called \emph{operational} vector fields, but they are not relevant for this work. Note that the Lie bracket $[-,-]$ is defined on operational vector fields, but it is a theorem in \cite{KrieglMichor97} that $\vect(M)$ is a Lie sub-algebra of this algebra.

\begin{definition}[Differential $k$-forms]
    A kinematic $k$-form is a section of the sub-bundle
    \[ L^k_{\mathrm{alt}}(TM, M\times \R) \subset \Cinf(TM\times_M ... \times_M TM, M).\]
    In particular for $M=E$, a vector space, a $k$-form is a map
    \[\alpha\colon E \times \bigwedge^k E \rightarrow \R\]
    which is smooth in the first copy of $E$ and linear in $\bigwedge^k E$. In other words, it is a map
    $\alpha\colon E \times E^k \rightarrow \R$ which is smooth in the first copy of $E$, and continuous, linear, and alternating in the other $k$ copies of $E$. $k$-forms on more general manifolds modelled by a convenient space $E$ can then be ``stiched together'' from an atlas as one would expect from the finite-dimensional setting. 
\end{definition}

This is, in fact, the only definition of $k$-forms for which many of the standard statements about $k$-forms generalise into infinite dimensions. This is the subject of \textsection 33 of \cite{KrieglMichor97}.

\subsubsection{Symplectic manifolds}
We briefly mention that symplectic geometry generalises readily into infinite dimensions.
\begin{definition}[Weakly symplectic vector space]
    Let $E$ be a TVS and $\Omega\colon E \times E \rightarrow \R$ be a continuous bilinear form.
    Define the map $\flat_\Omega\colon E \rightarrow E$ by
    \[\flat(v)(w) = \Omega(v,w).\]
    If $\ker \flat_\Omega = \braces{0}$ then we call $\Omega$ \emph{weakly non-degenerate}. If $\flat_\Omega$ is an isomorphism then we call $\Omega$ \emph{strongly non-degenerate}. Note that when the dimension of $E$ is finite, the concepts of weak and strong non-degeneracy coincide.

    If in addition, $\Omega$ is antisymmetric then we call the pair $(E, \Omega)$ a \emph{weakly (strongly) symplectic vector space}.
\end{definition}

Note that if $(E,\Omega)$ is strongly non-degenerate then $E$ must be self-dual. Hence if we restricted ourselves to strong non-degeneracy, even spaces such as $\ell^p$ would be beyond our scope (unless $p=2$).

\begin{definition}[Weakly symplectic manifold]
    Let $M$ be a manifold modelled by some TVS $E$, and let $\omega \in \Omega^2(M)$. We call $\omega$ weakly non-degenerate if the map $\flat \colon TM \rightarrow T'M$
    \[ \flat\colon v \rightarrow \iota_v \omega \]
    is an injective bundle-morphism. Equivalently, $(T_xM, \omega_x)$ is a weakly symplectic vector space for each $x\in M$. If $\flat$ is in addition an isomorphism then we call $\omega$ strongly non-degenerate.

    We may also consider $\flat$ as a map $\vect(M) \rightarrow \Omega^1(M)$, with $\flat(X):= \iota_X \omega$.

    The pair $(M,\omega)$ is then a \emph{weakly (strongly) symplectic manifold} if 
    \begin{itemize}
        \item $\omega$ is closed, i.e. $d\omega =0$,
        \item and $\omega$ is weakly (strongly) non-degenerate.
    \end{itemize}
\end{definition}

\begin{definition}[Hamiltonian vector field]
    Let $(M,\omega)$ be a (weakly) symplectic manifold modelled by a TVS $E$ and let $H \in \Cinf(M)$. Then a vector field $X$ is called a \emph{Hamiltonian vector field of $H$} if 
    \[\flat(X) = dH.\]
    Since $\flat$ is at least injective, if a hamiltonian vector field of $H$ exists, then it is unique and we refer to \emph{the} Hamiltonian vector field of $H$, denoted $X_H$. 

    If $(M,\omega)$ is strongly symplectic then we can define $\sharp := \flat^{-1} \colon \Omega^1(M)\rightarrow \vect(M)$ and $X_H$ always exists, being given by $X_H = \sharp(dH)$. 
\end{definition}
\section{Infinite-dimensional contact manifolds}
The standard definition of a (strict, co-orientable) contact manifold is the following:
\begin{definition}\label{def: finite contact}
    Let $M$ be a $2n+1$-dimensional manifold and $\eta \in \Omega^1(M)$. The pair $(M,\eta)$ is a strict, co-orientable manifold if 
    \[\eta \wedge (d\eta)^n > 0\]
    is a volume form (a nowhere zero form of top degree).
\end{definition}

Of particular interest is the study of contact Hamiltonian systems, which generalises the Hamiltonian vector fields of symplectic geometry into the case where the Hamiltonian need not be conserved. This allows us to use contact forms to model physical systems that have dissipative or irreversible properties, such as in thermodynamics \cite{Bravetti17}. The model example is the damped harmonic oscillator \cite[\textsection 7.1]{deLeonIzquierdo24}.

If we wish to extend the study of dissipative Hamiltonian systems to infinite dimensions then we should first generalise contact geometry into infinite dimensions. The first problem is that the above definition makes explicit use of the dimension of the manifold, and if we try to use this naïvely, we quickly run into issues, since the concept of a differential ``$\infty$-form'' is not defined.

We hence intoduce a definition of contact geometry which does not make explicit use of the dimension of the manifold, and so allows us to generalise the concept of a contact manifold to the infinite-dimensional setting. In particular, we introduce ``weakly'' contact manifolds in order to study contact geometry on manifolds modelled by spaces which are not self-dual. We justify this new definition by comparing it both with infinite-dimensional symplectic geometry and with finite-dimensional contact geometry. From now on we assume all manifolds to be modelled by a convenient space, which includes manifolds modelled by Banach or Fréchet spaces.

\begin{theorem}\label{thm: inf dim contact forms}
    Let $M$ be an infinite-dimensional manifold, and $\eta\in \Omega^1(M)$ be a 1-form such that $d\eta$ is degenerate (has non-zero kernel). We define the \emph{horizontal} and \emph{vertical} distributions (subbundles of $TM$) as $\Hor:= \ker \eta$ and $\Ver:= \ker d \eta = \ker (v \mapsto d\eta(v,-))$ respectively. The following are equivalent:
    \begin{enumerate}
        \item The tangent bundle is equal to the Whitney sum of the horizontal and vertical tangent bundles; i.e. $TM = \Hor \oplus \Ver$
        \item $d\eta$ is weakly non-degenerate on $\Hor$, i.e. $d\eta|_\Hor$ has zero kernel. Equivalently, $(\Hor, d\eta|_\Hor)$ is a weakly symplectic vector bundle.
        \item The map \[\flat\colon \vect(M)\rightarrow \Omega^1(M)\]
        \[ \flat(X) = \iota_X d\eta + \eta(X)\eta\]
        is an injective homomorphism of $\Cinf(M)$-modules. Equivalently, $\flat: TM \rightarrow T'M$ is an injective bundle homomorphism.
    \end{enumerate}
\end{theorem}

Note that for finite-dimensional $M$, conditions (1-3) are equivalent to definition \ref{def: finite contact}.

\begin{proof}
    We prove $(1)\Rightarrow(2)\Rightarrow(3)\Rightarrow(1)$.
    
    $(1)\Rightarrow(2)$: Assume that $TM = \Hor \oplus \Ver$ and let $p\in M$. Then $\Hor_p + \Ver_p = T_pM$ and $\Hor_p \cap \Ver_p = \braces{0}$. Then 
    \[\ker d\eta_p|_{\Hor_p} = \ker d\eta_p \cap \Hor_p = \Ver_p \cap \Hor_p = \braces{0}\]
    i.e. $d\eta|_\Hor$ is weakly non-degenerate.

    $(2)\Rightarrow(3)$: Assume that $d\eta|_\Hor$ is weakly non-degenerate. Then 
    \[ \ker d\eta|_\Hor = \ker \eta \cap \ker d\eta=\braces{0}.\] 

    To show that $\flat$ is injective, we show that it has zero kernel, hence let $X \in \vect(M)$ such that $\flat(X) = 0$. It follows that 
    \begin{align*}
        \flat(X) = \iota_X d\eta + \eta(X)\eta &=0\\
        \Rightarrow \eta(X)\eta = -\iota_X d\eta. &
    \end{align*}
    Then contracting both sides by $X$,
    \[\eta(X)\eta(X) = d\eta(X,X) = 0,\]
    therefore $\eta(X)\equiv 0$ and so $X$ is a section of $\Hor$.
    It hence follows that $\iota_X d\eta = 0$, and so $X$ is a section of $\Ver$. Hence, for all $p \in M$, $X_p \in \Hor_p \cap \Ver_p = {0}$, so $X=0$ identically and $\flat$ is injective.

    $(3)\Rightarrow (1)$: Assume that $\flat$ is an injective homomorphism of $\Cinf(M)$-modules. Then $\ker \flat = \braces{0}$. 

    Let $X \in \vect(M)$ such that $X$ is a section of both $\Hor$ and $\Ver$. Then 
    \[\flat(X) = \iota_X d\eta +\eta(X)\eta = 0,\]
    so $X\in \ker \flat$ and $X=0$ identically. We conclude that $\Hor \cap \Ver = \braces{0}$.
    
    Since $\Hor\cap\Ver = \braces{0}$, $\codim\Hor_p = 1$ for all $p$, and since $\dim d\eta>0 $ by assumption of degeneracy, it follows, by lemma \ref{codimension lemma} (``the codimension lemma") that $\dim \Ver_p = 1$ for all $p$. It hence follows by the second part of the lemma that $\Hor \oplus \Ver = TM$.
\end{proof}

\begin{definition}[Contact form\label{def: inf dim contact form}]
    Any 1-form $\eta$ as in the theorem above is a \emph{(weakly) contact form} if it satisfies any (hence all) of conditions (1)-(3).
\end{definition}

\begin{definition}[Contact manifold]
    Let $(M,\xi)$ be a manifold paired with some codimension-1 distribution, $\xi \subset TM$.

    If every point $p \in M$ has some neighbourhood $U$ such that, for some $\eta_U \in \Omega^1(U)$, $\xi|_U = \ker \eta_U$ and each $\eta_U$ is a contact form on $U$, then we call $\xi$ a \emph{contact distribution} and $(M,\xi)$ a \emph{contact manifold}.
\end{definition}

\begin{definition}
    A distribution $\xi \subset TM$ is called co-orientable if the quotient bundle, $TM/\xi$ is trivial. 
\end{definition}
\begin{theorem}
    Let $\xi\subset TM$ be a distribution on $M$. $\xi$ is co-orientable if and only if $\xi = \ker \alpha$ for some nowhere-zero 1-form $\alpha \in \Omega^1(M)$.
\end{theorem}
\begin{proof}
    This argument is lifted wholesale from \cite[p.3]{Geiges08}.

    Locally, the bundles $TM$ and $T'M$ are trivial, so over a small neighbourhood $U\subset M$ we may write a 1-form $\alpha_U$ defining $\xi|_U$ as the pullback of some non-zero section of $(TM/\xi)'$ by the projection $\pi: TM \rightarrow TM/\xi$. Clearly $\ker \alpha_U = \xi_U$.
    
    Assume that $\xi$ is co-orientable, so that, globally, $TM/\xi$ is trivial, and hence $(TM/\xi)'$ has a global non-zero section. Then construct $\alpha \in \Omega^1(M)$ as above.

    Conversely, if $\alpha$ is a nowhere-zero 1-form with $\xi = \ker \alpha$, then this induces a non-zero section of $(TM/\xi)'$ and in turn this implies that $TM/\xi$ is trivial.
\end{proof}

\begin{definition}
    If $\xi = \ker \eta$ for some contact form $\eta \in \Omega^1(M)$ then we call $(M, \xi)$ a \emph{co-orientable contact manifold} and $\xi$ a \emph{co-orientable contact distribution}. Any such $\xi$ defines a conformal class, $\braces{f\eta \, | \, f\in \Cinf(M), \, f>0}$, of contact forms which define $\xi$.

    Given a specific choice, $\eta$, of contact form for a co-orientable distibution we call $(M,\eta)$ a \emph{strict, co-orientable contact manifold}.
\end{definition}

\begin{remark}
    From now on a \emph{contact manifold} will always refer to a \emph{strict co-orientable contact manifold} and will be denoted $(M,\eta)$. We will also henceforth interchangeably refer to the contact distribution and the horizontal distribution, $\xi = \Hor  = \ker \eta$.
\end{remark}

\begin{definition}[Reeb vector field]
    Since $\Ver$ is a distribution of constant rank 1 (a line field) and $\eta$ is nowhere zero, there exists a unique vector field $R \in \vect(M)$ such that $\iota_R \eta = 1$ and $\iota_R d\eta = 0$. This vector field is the \emph{Reeb vector field}.
\end{definition}

\begin{theorem}
    $(\Hor, d\eta|_{\Hor})$ is a strongly symplectic distribution iff $\flat$ is an isomorphism $TM \rightarrow T'M$.
\end{theorem}

\begin{proof}
    $\Rightarrow$: Assume that $d\eta|_\Hor$ is bijective. Hence $\eta$ is a contact form by definition \ref{def: inf dim contact form} as long as $d\eta$ is degenerate on the whole manifold (otherwise $d\eta$ is a symplectic form). Let $\alpha \in \Omega^1(M)$, and $Z+fR$, $Z \in \Gamma(\Hor\rightarrow M)$, $f\in \Cinf(M)$, be a general vector field on $M$. Then $\alpha = \alpha_\Hor + \alpha_\Ver$ where
    \begin{align*}
        \alpha_\Hor(Z+fR) &= \alpha(Z) \\
        \alpha_\Ver(Z+fR) &= f\alpha(R) = \eta(Z+ fR)\alpha(R)
    \end{align*}
    Define $\sharp(\alpha) := X + \alpha(R)R$ where $X$ is the unique horizontal vector field satisfying $\iota_X d\eta = \alpha_\Hor$, which exists by assumption of bijectivity. Then 
    \begin{align*}
        (\flat\circ\sharp)(\alpha) &= \iota_{X + \alpha(R)R}d\eta + \eta(X + \alpha(R)R)\eta \\
        &= \iota_X d\eta + \alpha(R)\eta \\
        &= \alpha_\Hor + \alpha_\Ver \\
        &= \alpha
    \end{align*}
    and 
    \begin{align*}
        (\sharp \circ \flat)(Z+fR) &= \sharp(\iota_Z d\eta + f\eta)\\
        &=Z+fR
    \end{align*}

    $\Leftarrow$: Assume that $\flat$ is bijective. Let $\alpha \in \Omega^1(M)$. Then there exists a unique $(Z+fR) \in\vect(M)$ such that $\flat(Z+fR) = \alpha$. Then when restricting this to $\Hor$ we get $\alpha|_\Hor = \flat(Z) = \iota_Z d\eta|_\Hor$, so there exists a unique horizontal vector field $Z$ such that $\alpha|_\Hor = \iota_Z d\eta$. 
\end{proof}

\begin{definition}
    When $d\eta$ is strongly non-degenerate on $\Hor$, then we call $\eta$ a \emph{strongly contact form}.
\end{definition}

The above property implies that $M$ is modelled by some lcTVS $E$ with $E'\cong E$. Hence, proper Fréchet manifolds may never have strongly contact forms, since the only reflexive Fréchet spaces are Banach spaces.

We have generalised contact geometry to infinite-dimensional settings, but at the cost of our original definition, that $\eta\wedge(d\eta)^n$ is a volume form on $(M^{2n+1})$, since this depends on the dimension being finite. However, the following theorem provides us with an analogue. It also suggests that the connection between contact forms and Pfaffian problems persists in infinite dimensions.

\begin{theorem}\label{thm: infinite-dimensional analogue of definition}
    If $(M,\eta)$ is an infinite-dimensional contact manifold, then
    \[\eta\wedge (d\eta)^k \neq 0 \text{ for all }k\in \mathbb{N}.\]
\end{theorem}

The proof of this theorem first requires a lemma on symplectic vector spaces, which is a generalisation of theorem 1.1 of \cite{CannasdaSilva}.

\begin{lemma}\label{fin sim symplectic vector spaces}
    Let $(V,\Omega)$ be a weakly symplectic vector space, i.e. let $\Omega\colon V\times V \rightarrow \R$ be an antisymmetric bilinear form such that $\Omega(v,-)=0 \ \Rightarrow \ v=0$. For a subspace $W\subset V$ define
    \[ W^{\bot_\Omega}:=\{ v\in V | \ \Omega(v,w)=0 \ \forall w \in W \}. \]
    If $W$ is finite-dimensional and $(W, \Omega|_W)$ is a weakly symplectic space, then
    \begin{itemize}
        \item $V = W\oplus W^{\bot_\Omega}$ and 
        \item $(W^{\bot_\Omega}, \Omega|_{W^{\bot_\Omega}})$ is weakly symplectic.
    \end{itemize}
\end{lemma}
\begin{proof}
    Since $W$ is finite-dimensional, weakly symplectic implies strongly symplectic, so $W$ is even-dimensional and there exists a canonical basis $W = \angles{e_1, f_1,...,e_k,f_k}_\R$ such that 
    \begin{itemize}
        \item $\Omega(e_i,f_j) = \delta_{ij}$
        \item $\Omega(e_i, e_j) = \Omega(f_i, f_j)= 0$.
    \end{itemize}
    To show that $V = W\oplus W^{\bot_\Omega}$, first we show that $W \cap W^{\bot_\Omega} = \{0\}$, so pick $v\in W \cap W^{\bot_\Omega} $. Then $v = a_1 e_1 + b_1 f_1 +...+b_k f_k$ for some $a_1,...,a_k, b_1,...,b_k \in \R$. So,
    \begin{align*}
        0=&\Omega(e_i, v) = b_i \\
        0=& \Omega(f_i,v) = -a_i
    \end{align*}
    and $v=0$.
    Now to show that $V = W + W^{\bot_\Omega}$, let $v\in V$, and write
    \[v = \left(\sum_i \Omega(e_i, v)f_i - \Omega(f_i, v) e_i \right) + \left(v-\left(\sum_i \Omega(e_i, v)f_i - \Omega(f_i, v) e_i \right)\right), \]
    noting that the first term belongs to $W$ and the second to $W^{\bot_\Omega}$.

    $\Omega|_{W^{\bot_\Omega}}$ is antisymmetric, and it is weakly non-degenerate because, for any non-zero $v\in W^{\bot_\Omega}$, we have $\Omega(v,w)=0$ for all $w\in W$, so there must exist some $u\in W^{\bot_\Omega}$ such that $\Omega(v,u)\neq 0 $, otherwise $v=0$. Hence we conclude that $(W^{\bot_\Omega}, \Omega|_{W^{\bot_\Omega}})$ is weakly symplectic.
\end{proof}
 
The proof of theorem \ref{thm: infinite-dimensional analogue of definition} now follows readily.

\begin{proof}
    Fix $p\in M$, and $k\in \mathbb{N}$. $T_pM = \Hor_p \oplus \Ver_p$, so pick some element $v\in \Ver_p$, which we can assume satsfies $\eta(v)=1$. Now note that $(\Hor_p, d\eta|_{\Hor_p})$ is a weakly symplectic vector space. Choose some $e_1\in \Hor_p$. Then by non-degeneracy of $d\eta|_{\Hor_p}$ there is some $f_1 \in \Hor_p$ with $d\eta(e_1,f_1)=1$. Define $V_1 = \angles{e_1, f_1}$. $(V_1, d\eta|_{V_1})$ is a finite-dimensional symplectic vector space, so we can apply lemma \ref{fin sim symplectic vector spaces}, which tells us that $\Hor_p = V_1 \oplus V_1^{\bot_{d\eta|_{\Hor_p}}}$. Using induction, we can repeat this process finitely many times so that we have $2k$ vectors $e_1,f_1,...,e_k,f_k\in\Hor_p$ which form a canonical basis of a finite-dimensional subspace of $\Hor_p$. We then compute
    \begin{align*}
        \left(\eta\wedge (d\eta)^k\right)(v,e_1,f_1,...,e_k,f_k) &= \eta(v)\cdot d\eta(e_1,f_1) \dots  d\eta(e_k,f_k)\\
        &=1 \ \neq 0.
    \end{align*}
\end{proof}

An interesting question is to what degree the converse of this holds. That is, let $M$ be a manifold and $\eta\in\Omega^1(M)$ be a 1-form on $M$ such that for all $k\in \mathbb{N}$, $\eta \wedge (d\eta)^k\neq 0$, then which extra conditions on $M$ and $\eta$ make $\eta$ into a contact form? We note the following lemma as a first step towards this open question.
\begin{lemma}\label{lem: H cap V is involutive in general}
    Let $M$ be a manifold and $\eta \in \Omega^1(M)$. Define $\Hor:= \ker \eta$ and $\Ver:= \ker d\eta$ as above. Then $\Hor \cap \Ver$ is an involutive distribution.
\end{lemma}
\begin{proof}
    Let $X,Y \in \vect(M)$ such that $\eta(X) = \eta(Y) = 0$ and $\iota_X d\eta = \iota_Y d\eta = 0$.
    Then 
    \begin{align*}
        0 &=d\eta(X,Y) \\
        &= X.\eta(Y) - Y.\eta(X) - \eta([X,Y])\\
        &= -\eta([X,Y])
    \end{align*}
    and for arbitrary $Z \in \vect(M)$,
    \begin{align*}
        0 = dd\eta =& X.d\eta(Y,Z)-Y.d\eta(X,Z)+Z.d\eta(X,Y)\\
        &\qquad -d\eta([X,Y], Z) + d\eta([X,Z], Y) - d\eta([Y,Z], X) \\
         =& -d\eta([X,Y], Z).
    \end{align*}
    We therefore conclude that $\eta([X,Y])=0$ and $\iota_{[X,Y]}d\eta = 0$ for any such $X,Y$ and that $\Hor \cap \Ver$ is an involutive distribution.
\end{proof}

Finally, we note that the property that $\forall k,\ \eta\wedge (d\eta)^k \neq 0$ implies that the codimension of $\Ver$ is infinite. 

\section{Examples and constructions}

We now generalise some well-known methods of constructing contact manifolds and produce a few examples. In particular we discuss contact forms on spaces of the form $\R \times L^p(X,d\mu) \times L^q(X,d\mu)$.

\subsection{Contact manifolds via products}\label{sec: trivial bundles}

\begin{theorem}\label{thm: contact form via R-product}
    Let $(M, d\theta)$ be a weakly symplectic infinite-dimensional manifold, and consider the manifolds $(\R, dz)$, $(S^1, dz)$, both parametrised with the coordinate $z$, paired with respective nowhere zero 1-forms, which WLOG we assume are equal to $dz$. 
    
    Consider the manifolds $\R\times M$ and $S^1 \times M$ with projections $\pi_1, \ \pi_2$ onto the first and second factors respectively. Let 
    \[\eta = \pi_1^* dz + \pi_2^*\theta,\]
    be a 1-form on either $\R\times M$ or $S^1 \times M$, which by an understandable abuse of notation we will write $\eta = dz + \theta$.
    Then the manifolds $(\R\times M, \eta)$ and $(S^1 \times M, \eta)$ are contact manifolds wrt definition \ref{def: inf dim contact form}.
\end{theorem}
\begin{proof}
    The property of being a contact form can be checked locally, therefore we can assume WLOG that we are in the first case, with $(\R\times M, \eta)$, so pick any $(z,p)\in \R \times M$.

    We have \[d\eta = ddz + d\theta = d\theta,\] since the differential commutes with pull-backs (for details of this fact in the case of convenient manifolds, refer to \cite[\textsection 33]{KrieglMichor97}).
    Hence, for $(z,p)\in \R \times M$,
    \begin{align*}
        \Hor_{(z,p)} &= \{(\zeta, v) \in T_{(z,p)}(\R\times M) | \ \zeta + \theta(v) = 0\}\\
        &= \braces{(-\theta(v), v) | \ v \in T_p M},\\
        \\
        \Ver_{(z,p)}&= \braces{ (\zeta, v) \in T_{(z,p)}(\R\times M) |\ d\eta((\zeta, v) ,-)=0 }\\
        &= \braces{(\zeta,v) | \ d\theta(v,-) = 0}\\
        &= \braces{(\zeta, 0) | \zeta \in \R},
    \end{align*}
    where the last equality follows from the fact that $d\theta$ is weakly non-degenerate. We can therefore conclude that for all $(z,p)\in \R\times M$, $\Hor_{(z,p)}\cap \Ver_{(z,p)} = 0$, and hence that $TM = \Hor \oplus \Ver$, so $\eta$ satisfies the conditions of definition \ref{def: inf dim contact form}. As an exercise, the reader is invited to check the other two conditions of definition \ref{def: inf dim contact form} directly. 

    It should be clear that for any local representation of $U\subset M$, the local representation of the Reeb vector field is $R = (1,0) =: \frac{\partial}{\partial z}$.
\end{proof}

The following example is a modification of an example of a weak symplectic Fréchet space given in \cite{DiezRudolph24}.

\begin{example}[A contact Fréchet manifold]\label{ex: R times Omega^1 M}
    Let $(M,g)$ be a closed orientable 2-dimensional Riemannian manifold, with metric $g$ inducing a volume form $\vol \in \Omega^2(M)$, and hence a Hodge star operator $*: \Omega^1(M) \rightarrow \Omega^1(M)$ such that $\vol = \alpha \wedge * \alpha$ for any $\alpha \in \Omega^1(M)$. We also denote by $\alpha^{\sharp_g}$ the unique vector field satisfying $g(\alpha^{\sharp_g}, -) = \alpha$. This defines an inner product $(-|-)$ on $\Omega^1(M)$ by
\[ (\alpha|\beta):=\int_M g(\alpha^{\sharp_g}, \beta^{\sharp_g}) = -\int_M * \alpha \wedge \beta. \]
Refer to \cite[\textsection 9.4]{JMLee09} for details.

Now take the Fréchet space $\mathfrak{M} = \R \times \Omega^1(M)$, which we will consider as a manifold modelled by itself. We will denote $p \in \mathfrak{M}$ by the pair $p = (z, \alpha)$ and we will abuse notation to represent tangent vectors by
\[ \zeta \deez + \beta \in T_{(z,\alpha)}\mathfrak{M} \cong \mathfrak{M}, \] 
where $\zeta \in \R$ and $\alpha \in \Omega^1(M)$. Further tangent vectors will be denoted $\theta \deez + \gamma$. We should interpret mixed expressions such as $dz(\alpha)$ and $\alpha \wedge \deez$ to be equal to 0, in order to make our notation consistent. Then $(\mathfrak{M}, \eta)$ is a contact manifold, where
\[\eta_{(z,\alpha)}(\zeta + \beta) = \zeta + \frac{1}{2}\int_M \alpha \wedge \beta,\]
or, 
\[\eta_{(z, \alpha)} = dz + \frac{1}{2}\int_M \alpha \wedge -.\]

\end{example}
\begin{proof}
    It suffices, due to the above theorem, to check that $(\Omega^1(M), \omega)$ is a symplectic form when $\omega = d\theta$ and, for $\beta \in T_\alpha(\Omega^1(M)) \cong \Omega^1(M)$,
    \[\theta_\alpha(\beta) = \frac{1}{2}\int_M \alpha \wedge \beta.\]

    To evaluate $d\theta_\alpha(\beta, \gamma)$ for $\beta, \gamma \in T_\alpha(\Omega^1(M)) \cong \Omega^1(M)$, extend $\beta$ and $\gamma$ to constant vector fields (so that $[\beta, \gamma]=0$), and write
    \begin{align*}
        d\theta(\beta, \gamma) &= \beta. \theta(\gamma) - \gamma. \theta(\beta)\\
        &= \frac{1}{2}\left. \frac{d}{dt} \right|_{t=0}\int_M (\alpha + \beta t) \wedge \gamma - (\alpha + \gamma t) \wedge \beta \\
        &= \frac{1}{2}\int_M \beta \wedge \gamma - \gamma \wedge \beta = \int_M \beta \wedge \gamma,  \\
    \end{align*}
    which is weakly non-degenerate, since for any Riemannian metric $g$ with an associated Hodge star operator $*$, $d\theta(\alpha, *\alpha) = 0$ iff $\alpha=0$.
\end{proof}

\begin{example}[The jet bundle of a Banach manifold]
    Let $M$ be a Banach manifold modelled on a Banach space $(E, \norm{-})$, and consider its cotangent bundle $T'M$. We let $\angles{w,\xi}$ be the natural pairing between vectors $w\in T_xM$ and bounded covectors $\xi \in T'_xM := (T_xM)'$. Let $p = (x, \xi)$ be a point belonging to $T'M$, and then define the 1-form $\theta \in \Omega^1(T'M)$ by
    \[ \theta_p(v) = \angles{T\pi \cdot v, \xi}. \]
    This is the \emph{canonical 1-form} on the cotangent bundle, whose differential, $\omega = d\theta$ is the \emph{canonical 2-form}, which is a weak symplectic form by the Hahn-Banach theorem. For more details on this construction refer to \cite{Lang72}.

    We can hence take the pair $(\R\times T'M, \ \eta = dz+ \theta)$ which is, by the above theorem, a contact Banach manifold. This is isomorphic to the first jet bundle of smooth functions on $M$, $J^1(M)$, as defined in \cite[\textsection 1.4]{Dodson15}.
\end{example}

\begin{example}\label{ex: L^p spaces}
    A special case of the above is where $M=E$. Then we can write the contact form $\eta\in\Omega^1(\R\times E \times E')\ $ as 
    \[\eta_{(z, x, \varphi)}(\zhat, \xhat, \phihat) = \zhat + \varphi(\xhat),\]
    where $(\zhat, \xhat, \phihat)$ is a tangent vector with foot-point $(z, x, \varphi)\in \R\times E \times E'$.

    For a concrete example, let $E = L^p(X, d\mu)$ with the $L^p$ norm. Then the manifold of interest is $\R\times L^p(X,d\mu) \times L^q(X,d\mu)$ where $p$ and $q$ satisfy $pq=p+q$. Let $(z, f, g)\in \R\times L^p(X,d\mu) \times L^q(X,d\mu)$, and let $(\hat{z}, \hat{f}, \hat{g})$ be a vector tangent to this point. Then 
    \[\eta_{(z, f, g)}(\hat{z}, \hat{f}, \hat{g}) = \hat{z} + \int_X \hat{f}g \ d\mu \]
    defines a contact form. 
    
    Further, when $p=q=2$, by the Riesz representation theorem, this is in fact a strong contact form, since $d\eta$ is strongly symplectic when restricted to $\Hor \cong L^2(X,d\mu)$.
\end{example}

\subsubsection{Canonical contact structures beyond Banach spaces}
The jet bundle construction as above may also take place on convenient manifolds, since the class of convenient vector spaces is large enough that the (bounded) dual of a convenient vector space is once again convenient (for jet bundles on convenient manifolds refer to \cite[\textsection 41]{KrieglMichor97}).

This however does not hold in Fréchet spaces; the only Fréchet spaces whose duals are also Fréchet are Banach spaces. If we wish to consider canonical contact structures with configuration spaces modelled by properly Fréchet spaces, then we can use the following construction, given by Diez and Rudolph in \cite[A.1]{DiezRudolph20}.

\begin{definition}[Dual pairing between vector bundles]
    Let $\pi\colon E \rightarrow M$ and $\rho\colon F \rightarrow M$ be two vector bundles. Assume that all manifolds are Fréchet and let $h\colon E \times_M F \rightarrow \R$ be a bilinear form. We call $h$ a \emph{dual pairing} between $E$ and $F$ if it is fibrewise non-degenerate. Now if $F = TM$ then we call $E$ \emph{a cotangent bundle} of $M$ if such a dual pairing exists and we write $T'Q \cong E$.
\end{definition}
Given a dual pairing the authors then note that, for $p = (x, \xi) \in T'M$ and $v \in T_p(T'M)$ the expression
\[ \theta_p(v) = \angles{T\pi \cdot v, \xi} \]
defines a 1-form on $T'M$ and that $d\theta$ is weakly symplectic. It hence follows from theorem \ref{thm: contact form via R-product} that $dz + \theta$ is a contact form for the manifold $\R \times T'M$.

\subsection{Submanifolds of contact type}
\begin{theorem}[Level sets of contact type]\label{thm: level sets of contact type}
    Let $(M,\omega)$ be a weakly symplectic manifold, and fix some $H\in \Cinf(M)$, called the Hamiltonian. Define $N:=H^{-1}(0)$, and make the following assumptions:
    \begin{itemize}
        \item $N\subset M$ is a submanifold, with inclusion $i\colon N \hookrightarrow M$,
        \item there exists some $Y\in \vect(M)$ such that: \begin{itemize}
            \item $Y$ is transverse to $N$, that is, for all $x\in N$, $T_xM = T_xN \oplus \angles{Y_x}$
            \item $\mathcal{L}_Y \omega = f\omega$ for some $f \in \Cinf(M)$ which is nowhere zero (at least, we assume this holds in a neighbourhood containing $N$), equivalently, $d\iota_Y \omega = f\omega$,
        \end{itemize}
        \item there exists a (necessarily unique) $X_H\in \vect(M)$ such that $\iota_{X_H}\omega = dH$.
    \end{itemize} 
    Then $(N, i^*(\iota_Y \omega))$ is a contact manifold.
\end{theorem}
\begin{proof}
    We have
    \begin{align*}
        \Hor_x &= \{ v \in T_xN | \ \omega(Y_x, v) = 0 \}\\
        \Ver_x &= \{ v \in T_xN | \ \forall w\in T_xN,\ \omega(v, w) = 0 \}.
    \end{align*}
    Since $\Hor_x$ is the kernel of a non-zero linear map, it must have codimension 1. We also have
    \begin{align*}
        \Hor_x \cap \Ver_x &= \{ v \in T_xN | \ \omega(Y_x, v) = 0, \ \omega(v, w) = 0  \ \forall w\in T_xN\}\\
        &= \{ v\in T_xN | \ \omega(v, \, aY_x + w)=0, \ \forall a \in \R, \, \forall w \in T_xN \} \\
        &= \{ v\in T_xN | \ \omega(v, \, u)=0, \ \forall u\in T_xM \ \}\\
        &=\{0\},
    \end{align*}
    because $\omega_x$ is weakly non-degenerate on $T_xM$. By the codimension lemma, this forces $\Ver_x$ to have dimension either 1 or 0. However we know that it is not zero-dimensional since $\omega_x(X_H, Y_x) = dH(Y_x) \neq 0$ (since otherwise, $Y$ would be tangent to $N$ at this point) and for all $w\in T_xN$, $\omega(X_H, w) = dH(w)= 0$.
\end{proof}

\begin{example}
    Let $M$ be a closed orientable 2-manifold with Riemannianian metric $g$ associated with a Hodge star operator $*\colon \Omega^1(M) \rightarrow \Omega^1(M)$. Define $\mathfrak{M}:= \Omega^1(M)$ which we will consider as a Fréchet manifold modelled by itself. Throughout, we identify $T_\alpha \mf{M} \cong \mf{M}$ for all $\alpha \in \mf{M}$.

    $\mf{M}$ has a weakly symplectic form given by
    \[ \omega(\beta, \gamma) = \int_M \beta \wedge \gamma, \]
    and we will fix $\mf{M}$ to have Hamiltonian $H\colon\mf{M}\rightarrow \R$ given by 
    \[ H(\alpha) = \frac{1}{2}(\alpha|\alpha) =-\frac{1}{2}\int_M *\alpha \wedge \alpha. \]
    
    Define $\mf{N}:=H^{-1}(1)$. Then $\mf{N}\subset \mf{M}$ is a level set of contact type with contact form $\eta_\alpha(\beta) = \int_M \alpha \wedge \beta$.

\end{example}
\begin{proof}
    First we show that $\mf{N}$ is indeed a submanifold of $\mf{M}$. By \cite[Thm 27.11]{KrieglMichor97}, it suffices to find a local chart for each $\alpha$, from a closed subspace of $\Omega^1(M)$ into a neighbourhood $\alpha \in U \subset \mf{N}$.
    
    We proceed by essentially generalising orthographic projections for this ``unit sphere". Fix $\alpha_0 \in \mf{N}$. Define the map
    \begin{align*}
        \alpha\colon \left\{\beta \in \angles{\alpha_0}^\bot \ \middle| \ \int_M \beta \wedge *\beta < 1 \right\} \rightarrow \mf{N} \cap \{\alpha \ | \ (\alpha,\alpha_0)>0 \}\\
        \alpha(\beta) = \beta + \alpha_0 \sqrt{1-\int_M \beta \wedge *\beta},
    \end{align*}
    where $\bot$ is with respect to the inner product $(-|-)$. This map has inverse
    \[ \beta(\alpha) = \alpha - \alpha_0 \cdot \int_M \alpha \wedge *\alpha_0, \]
    and so it is a bijection between a subset of $\angles{\alpha_0}^\bot$ and an open subset of $\mf{N}$.
    It is easily checked that the domains of these maps are open, and that the maps themselves are indeed smooth, so that $\mf{N}$ is indeed a submanifold of $\mf{M}$.
    
    $X_H$ exists and is given by the map $ X_H \colon \alpha \mapsto -*\alpha$ since 
    \[ dH_\alpha (\beta) = \frac{d}{dt}\, \frac{1}{2}(\alpha + \beta t | \alpha + \beta t) = (\alpha | \beta) = - \int_M * \alpha \wedge \beta  = \omega(X_H, \beta).\]

    Define $Y\in \vect(M)$ as $Y_\alpha = \alpha$.

    $Y$ is transverse to $\mf{N}$ since $dH(Y) = -\int_M *\alpha \wedge \alpha = 1$ on $\mf{N}$, and 
    \[ d(\iota_Y \omega) = d\left( \int_M \alpha \wedge - \right) = \omega.\]
    
    Hence all of the assumptions of theorem \ref{thm: level sets of contact type} are satisfied, and $(\mf{N}, \iota_Y \omega|_{\mf{N}})$ is a contact manifold.
    
    The spaces $T_\alpha \mf{N}$, $\Hor_\alpha$, $\Ver_\alpha$ are as follows:
    \begin{align*}
        T_\alpha \mf{N} &= \ker dH = \left\{\beta \in \Omega^1(M) \ \middle | \ \int_M *\alpha \wedge \beta = 0 \right\} \\
        &= \angles{\alpha}^\bot, \\
        \Hor_\alpha &= \left\{\beta \in T_\alpha \mf{N} \ \middle| \ \int_M \alpha \wedge \beta = \int_M *\alpha \wedge \beta = 0 \right\} \\
        &= \angles{\alpha}^\bot \cap \angles{*\alpha}^\bot, \\
        \Ver_\alpha &= \angles{*\alpha}.
    \end{align*}
    We also have the following short exact sequence
    \[
    \begin{tikzcd}
        0 \arrow[r] & \angles{*\alpha} \arrow[r, "i", hook] & T_\alpha \mf{N} \arrow[r, "\pi", two heads] &  \arrow[r] \angles{\alpha}^\bot \cap \angles{*\alpha}^\bot & 0
    \end{tikzcd},
    \]
    where $i$ is the inclusion map and $\pi$ is the projection map given by
    \[ \pi\colon \beta \mapsto \beta-\left(\int_M *\alpha \wedge \beta \right) \alpha. \]
\end{proof}
\section{Contact Hamiltonian dynamics}
One of the main motivations for the study of contact geometry in finite dimensions is that it allows us to model dissipative and non-conservative physical systems. However, it is easy to think of many examples of such systems that are infinite-dimensional. For one, when a guitar string is plucked, we observe that it loses energy to its surroundings via sound energy, and heat energy.

We begin with a slight generalisation of Hamiltonian vector fields, and apply it to a toy example, describing its Hamiltonian flow. We end with the derivation of a Hamiltonian functional that describes the damped wave equation.

\subsection{Hamiltonian vector fields}
\begin{definition}[Contactomorphisms and contact vector fields]
    A diffeomorphism $F: (M,\eta) \rightarrow (N, \eta')$ is a \textit{contactomorphism} if it preserves the contact distributions induced by $\eta$ and $\eta'$ (and not, notably, the contact form itself), i.e. if the pushforward of the horizontal distribution $\Hor_M \subset TM$ induced by $\eta$ under $F$ is a sub-bundle of the horizontal distribution $\Hor_N \subset TN$ induced by $\eta'$. Equivalently, $F^*\eta = f\eta'$ for some $f\in\Cinf(M)$. We call $F$ a \textit{strict contactomorphism} if $F^*\eta = \eta'$.

    If $F_t, \ t \in (0,1)$ is a one-parameter family of diffeomorphisms $F_t: M\rightarrow M$, and each $F_t$ is a contactomorphism then ${F_t}_* \Hor = \Hor$. Equivalently, $F_t^*\eta = f_t\eta$ for some $f_t \in \Cinf(M)$.

    If we let $X$ be the vector field that represents an infinitesimal transformation by $F_t$ (so that $F_t$ is the flow of $X$) then we call $X$ a \textit{contact vector field}. Equivalently $X \in \vect(M)$ is \textit{contact} if 
    \[ \Lie_X \eta = f\eta \]
    for some $f \in \Cinf(M)$.
\end{definition}
The following definition is a straightforward generalisation of those found in \cite{Bravetti17,deLeon19, Montaldi24}, among others.
\begin{definition}[Hamiltonian and Hamiltonian vector field]\label{def: Hamiltonian v.f.}
    Given a contact vector field $X$ we associate to it the function $H_X = -\eta(X) \in \Cinf(M)$ called the \emph{Hamiltonian}.

    Given a function $H \in \Cinf(M)$ we call a vector field $X_H \in \vect(M)$ a \emph{Hamiltonian vector field} of $H$ if for some $f \in \Cinf(M)$,
    \begin{equation}\label{Hamiltonian v.f. conditions}
    \begin{cases}
        \eta(X_H) &= -H \\
        \Lie_{X_H}\eta &= f\eta.\\
    \end{cases}
    \end{equation}
\end{definition}
\begin{theorem}\label{Uniqueness of X_H}
    Let $H\in \Cinf(M)$. If a Hamiltonian vector field of $H$ exists then it is unique. Also, in the definition above, $f = -R(H)$.
\end{theorem}
\begin{proof}
    To prove the second statement, contract both sides of the second condition in equations \ref{Hamiltonian v.f. conditions} by $R$. Then
    \begin{align*}
        \Lie_{X_H}\eta(R) &= f\eta(R)\\
        d\eta(X_H, R) + d\eta(X_H)(R) &= f \\
        -dH(R) &= -R(H) = f.
    \end{align*}
    To prove the first, consider $\flat(X_H)$:
    \begin{align*}
        \flat(X_H) &= \iota_{X_H} d\eta + \eta(X_H)\eta \\
        & = \Lie_{X_H}\eta - d(\iota_{X_H}\eta) - H\eta  \\
        & = -R(H)\eta + dH - H\eta \\
        & = dH - (R(H)+H)\eta. \tag{$\ast$}
    \end{align*}
    By the injectivity of $\flat$, it is clear that if $X_H$ satisfies ($\ast$) then it is unique.
\end{proof}
\begin{remark}
    It follows from a simple calculation that 
    \[dH(Y) = d\eta(X_H, Y) + R(H)\eta(Y)\]
    and in particular, along the flow of $X_H$ (where it exists)
    \[\frac{d}{dt}H = X_H(H) = dH(X_H) = -R(H)H.\]
    Now suppose that there exists a chart in which the Reeb field is constant and equal to $\frac{\dee}{\dee z}$ for some co-ordinate $z$ (as in section \ref{sec: trivial bundles}). If $H$ is a Hamiltonian
    \[H = H_0 + f(z)\]
    where $H_0$ is independent of $z$ then the above becomes 
    \[\dot{H}(t) = f'(z(t)) H(t)\]
    which is seperable (enough), so that
    \[ H(t) = H(0) \cdot \exp{\left( -\int_0^t f'(z(t'))\, dt' \right)}.\]

    In the simplest examples, $f'(z)\equiv \kappa$ and 
    \[H(t) = H(0)e^{-\kappa t}.\] 
\end{remark}

In finite dimensions (or in the strongly non-degenerate case), the above proof would also guarantee existence of $X_H$, but here it only guarantees uniqueness. 
Any contact vector field $X$ has some Hamiltonian given by $H_X := -\eta(X)$, but as opposed to the finite-dimensional case, 
a Hamiltonian has a corresponding vector field iff $dH$ falls in the image of $\flat\colon \vect(M)\rightarrow \Omega^1(M)$. In \cite{DiezRudolph24} the authors give a linear Hamiltonian, $h\colon E\rightarrow \R$ as an example, then $X_h$ only exists as long as $h$ is in the image of $\flat$.

\subsection{A contact Hamiltonian system on a Fréchet space}
Let $(M,g)$ be a closed and orientable 2-dimensional Riemannian manifold. Recall (from example \ref{ex: R times Omega^1 M}) that $(\mathfrak{M}, \eta)$ is a contact manifold with $\mathfrak{M} = \R \times \Omega^1(M)$ and 
\[\eta_{(z, \alpha)} = dz + \frac{1}{2}\int_M \alpha \wedge -.\]

Now let $\kappa \geq 0$ and consider the Hamiltonian function $H \in \Cinf(\mathfrak{M})$,
\begin{align*}
H(z,\alpha) &= \kappa z + \frac{1}{2}(\alpha|\alpha)\\
&= \kappa z - \frac{1}{2}\int_M *\alpha \wedge \alpha.
\end{align*}

Our objective is to calculate the Hamiltonian vector field of $H$. 
To this end we write $X_H = \zeta_H\deez + \beta_H$ where each component varies over $\mathfrak{M}$. Recall that by theorem \ref{Uniqueness of X_H}, 
\[\flat(X_H) = dH - (R(H)+H)\eta .\]
Contracting the left-hand side of this equation by some $Y = \theta \deez + \gamma \in \vect(\mathfrak{M})$, we obtain 
\begin{align*}
    \flat(X_H)(Y) &= d\eta(X_H, Y) + \eta(X_H)\eta(Y) \\
    &= \int_M \beta_H \wedge \gamma + \eta \left( X_H \right) \left( \theta + \frac{1}{2}\int_M \alpha \wedge \gamma \right) \\
    &= \theta \cdot \eta \left( X_H \right) + \int_M \left( \beta_H + \frac{1}{2}\eta\left( X_H \right)\alpha \right)\wedge \gamma. \tag{$*$}
\end{align*}

The right-hand side requires more work. First, we calculate 
\[ dH_{(z,\alpha)} \left( \theta \deez + \gamma \right) = \kappa \theta - \int_M *\alpha \wedge \gamma, \]
and then since $R = \deez$ we have $R(H) = dH \left(\deez\right) = \kappa$, so

\begin{align*}
    (dH-(R(H)+H)\eta)&\left( \theta \deez + \gamma \right) \\ 
    &= \kappa\theta - \int_M *\alpha \wedge \gamma -\left(\kappa + H\right) \left( \theta + \frac{1}{2}\int_M \alpha \wedge \gamma \right) \\
    &= \theta \left( -H\right) + \int_M{\left(-*\alpha -\frac{1}{2}\left( \kappa + H \right)\alpha \right)\wedge \gamma}. \tag{$**$}
\end{align*}
Equating $(*)$ and $(**)$, we obtain the following pair of equations:
\begin{align*}
    \eta(X_H) &= -H \\
    \beta_H + \frac{1}{2}\eta\left(X_H\right)\alpha &= -*\alpha -\frac{1}{2}\left( \kappa + H \right)\alpha.
\end{align*}
Since $\eta(X_H) = -H$, we can rearrange the second equation into
\[ \beta_H = -*\alpha - \frac{1}{2}\kappa \alpha, \]
and the first then fixes $\zeta_H = -\kappa z$, so
\[ X_H = -\kappa z \deez + \left(-*\alpha -\frac{1}{2}\kappa\alpha\right).\]
The reader is invited to check that this satisfies the conditions of definition \ref{def: Hamiltonian v.f.}.

\subsubsection{Hamiltonian flow of $(\mathfrak{M}, \eta, H)$}
Now consider the differential equation 
\[ (\dot{z}, \dot{\alpha}) = X_H,\]
which is really two differential equations:
\begin{align}
    \dot{z} &= -\kappa z \label{equ: 1-form example z-component} \\
    \dot{\alpha} &= -*\alpha -\frac{1}{2}\kappa \alpha. \label{equ: 1-form example alpha-component}
\end{align}
Equation \ref{equ: 1-form example z-component} is solved by 
\[ z(t) = z(0)\cdot e^{-\kappa t}\]
and it doesn't directly contribute to any of the dynamics on $\Omega^1(M)$.

To solve equation \ref{equ: 1-form example alpha-component}, let  $q \in M$ and $x^1, y^1$ be positively oriented local co-ordinates for $M$ centred on $q$. Then $ \vol_q = \sqrt{\det g_{ij}} \ dx^1 \wedge dx^2$ and 
\begin{align*}
    *dx^1 &=\sqrt{\det g_{ij}} (-g^{12} dx^1 + g^{11}dx^2), \\
    *dx^2 &=\sqrt{\det g_{ij}} (-g^{22} dx^1 + g^{12}dx^2) .
\end{align*}
 Note that in two dimensions, $** = -\id$, so we may think of $*$ as an almost complex structure on $\Omega^1(M)$.

Locally $\alpha = \alpha_i dx^i$ and so we can rewrite equation \ref{equ: 1-form example alpha-component} as two coupled equations in $\alpha_1, \alpha_2$ as
\[
\begin{cases}
    \dot{\alpha}_1 = \sqrt{\det g_{ij}} \ (-g^{12} \alpha_1 - g^{22}\alpha_2) -\frac{1}{2}\kappa \alpha_1 \\
    \dot{\alpha}_2 = \sqrt{\det g_{ij}} \ (g^{11}\alpha_1 + g^{12}\alpha_2) -\frac{1}{2}\kappa \alpha_2, 
\end{cases}
\]
or, as
\begin{align*}  
\begin{pmatrix}
    \dot{\alpha}_1 & \dot{\alpha}_2
\end{pmatrix} & = -
\begin{pmatrix}
    \alpha_1 & \alpha_2
\end{pmatrix} 
\begin{pmatrix}
    -g^{12} & g^{11} \\
    -g^{22} & g^{12}
\end{pmatrix} \sqrt{\det g_{ij}} - \frac{1}{2} \kappa \begin{pmatrix}
    \alpha_1 & \alpha_2
\end{pmatrix}\\
& = - \begin{pmatrix}
    \alpha_1 & \alpha_2
\end{pmatrix} \braces{ 
\begin{pmatrix}
    g^{11} & g^{12} \\
    g^{12} & g^{22}
\end{pmatrix}
\begin{pmatrix}
    0 & \sqrt{\det g_{ij}} \\
    -\sqrt{\det g_{ij}} & 0
\end{pmatrix}
- \begin{pmatrix}
    \kappa/2 & 0 \\
    0 & \kappa/2
\end{pmatrix}
} \\
&= \begin{pmatrix}
    \alpha_1 & \alpha_2
\end{pmatrix}
\begin{pmatrix}
    g^{12} \sqrt{\det g_{ij}} - \kappa/2 & -g^{11}\sqrt{\det g_{ij}}\phantom{-\kappa/2} \\
    g^{22} \sqrt{\det g_{ij}}\phantom{-\kappa/2} & -g^{12}\sqrt{\det g_{ij}}-\kappa/2
\end{pmatrix}.
\end{align*}

The matrix that appears has been postmultiplied, and not premultiplied, because we are working with covectors and not vectors. This matrix has determinant
\begin{align*}
    \left( g^{12} \sqrt{\det g_{ij}} - \frac{\kappa}{2} \right) & \left( -g^{12}\sqrt{\det g_{ij}}-\frac{\kappa}{2} \right) - \left( -g^{11}\sqrt{\det g_{ij}}\right) \left( g^{22} \sqrt{\det g_{ij}}\right) \\
     &= g^{11}g^{22} \det g_{ij} - g^{12}g^{21} \det g_{ij}+\frac{\kappa^2}{4} \\
     &= \det g_{ij} \det g^{ij} + \frac{\kappa^2}{4} \quad = \det g_{ij}g^{jk} + \frac{\kappa^2}{4}\\
     &  = 1+ \frac{\kappa^2}{4}.
\end{align*}

Using the subsitution $\kappa \mapsto \kappa+2\lambda$ we can quickly write down the characteristic polynomial in $\lambda$ and hence the eigenvalues,
\begin{align*}
    \lambda^2 + \kappa \lambda + \left( 1+ \frac{\kappa^2}{4}\right) = 0 \quad \Rightarrow \lambda = -\frac{\kappa}{2} \pm i.
\end{align*}

We can hence write down the solution to equ. \ref{equ: 1-form example alpha-component} in local coordinates as 

\[ \begin{pmatrix}
    \alpha_1 (t) & \alpha_2 (t)
\end{pmatrix}  = \begin{pmatrix}
    \alpha_1 (0) & \alpha_2 (0)
\end{pmatrix}  \cdot
\begin{pmatrix}
    \cos t & - \sin t \\
    \sin t & \phantom{-}\cos t
\end{pmatrix} \cdot e^{-\kappa t /2}.
\]

Physically, for some $x \in M$, one could imagine the value of $\alpha_x(t)$ as a covector rotating with angular frequency $2\pi$ in the $T_x^*M$ plane, while simultaneously decaying in ``length'' at a rate of $\kappa/2$. 

\begin{remark}
    Given a point $q\in M$, this infinite-dimensional contact Hamiltonian system can be simplified to a 3-dimensional system $(\R^3\cong \R \times T^*M, \eta_{\text{std}}, h)$, where
    \[ h(z,x,y) = \kappa z + \frac{1}{2} (x^2+ y^2),\]
    the damped harmonic oscillator.
    Essentially the system on $\mathfrak{M}$ simultaneously solves a copy of this system at each $q \in M$. We conjecture that this relates to contact reduction on this system: consider the action on $\mathfrak{M}$ of the group of volume-preserving diffeomorphisms of $M$. Perhaps this reduced space is the copy of $\R^3$ above, but this is beyond the scope of this paper.
\end{remark}

\subsection{The Damped Wave Equation} \label{sec: damped wave equation}
Recall that in order to go from the simple harmonic oscillator as a Hamiltonian system, to the damped harmonic oscillator as a contact Hamiltonian system we simply pass from the phase space $T^*\R$ with Hamiltonian
\[H = \frac{1}{2}q^2 + \frac{1}{2}p^2\]
to the extended phase space $\R \times T^*\R$ with Hamiltonian 
\[H = \frac{1}{2}q^2 + \frac{1}{2}p^2 + \kappa z.\]

One obvious infinite-dimensional analogue to the simple harmonic oscillator is an infinitely long string with height function $u$ whose dynamics obey the wave equation
\[\frac{\dee^2 u}{\dee t^2} = \frac{\dee^2 u}{\dee x^2}.\]

Let $E=W^2_2(\R)$ be the space of $L^2$ functions on $\R$ whose first and second derivatives are also $L^2$. We consider $T'E$ to have the canonical symplectic structure. This is a Banach space; refer to \cite{AdamsFournier}. We may reproduce solutions to the wave equation as Hamiltonian flows induced by 
\[ H = \int_\R \frac{1}{2}\left( \frac{\dee u}{\dee x} \right)^2 + \frac{1}{2}p^2 \]
where $q=u$ and $p = \frac{\dee u}{\dee t}$.

One would then hope to use the same trick as above (extending the phase space, replacing the canonical symplectic structure with the canonical contact structure, and adding $\kappa z$ to the Hamiltonian) to reproduce solutions to the damped wave equation,
\[ \frac{\dee ^2 u}{\dee t^2} = \frac{\dee^2 u}{\dee x^2} - \kappa\frac{\dee u}{\dee t}, \]
as the flow of a contact Hamiltonian system.

Indeed, one can.

Consider the space $\R \times E \times E'$. We will write points of the space as $(z,q,p)$ and tangent vectors as $(\hat{z}, \hat{q}, \hat{p})$ or $(\tilde{z}, \tilde{q},\tilde{p})$. We use the same contact structure as in the $L^2$ case, namely,
\[ \eta_{(z,q,p)}(\hat{z},\hat{q},\hat{p}) = \hat{z} - \int_\R p \, \hat{q}\]
for which we also have 
\[ d\eta ( (\hat{z}, \hat{q}, \hat{p}) ,\, (\tilde{z}, \tilde{q},\tilde{p}) ) = \int_\R  \hat{q} \, \tilde{p} - \hat{p}\, \tilde{q}. \]
We have supressed the subscript in the second expression because $d\eta$ is constant with respect to the natural parallelisation of $\R \times E \times E'$. $W^2_2(\R)$ is a closed subspace of $L^2(\R)$, and although the topologies on the two spaces are induced by different norms, injectivity of $\flat$ on $L^2(\R)$ implies the same property for $W^2_2(\R)$ and we can check that $\im \flat \subset \R \times W^2_2(\R) \times W^2_2(\R)$. 

We then fix
\[ H = \kappa z + \int_\R \frac{1}{2}\left( \frac{\dee q}{\dee x} \right)^2 + \frac{1}{2}p^2 \ dx \]
to be the Hamiltonian of our system. Clearly, $R(H) = \frac{\dee H}{\dee z} = \kappa$. To apply theorem \ref{Uniqueness of X_H} we also need an expression for $dH$. By differentiating \[\frac{d}{dt}H(z+\hat{z}t,\, q+ \hat{q}t ,\, p+ \hat{p}t)\]
we obtain the expression
\[ dH_{(z,q,p)}(\hat{z}, \hat{q}, \hat{p}) = \kappa \hat{z} + \int_\R \frac{\dee q}{\dee x} \frac{\dee \hat{q}}{\dee x} + p\,\hat{p} \ dx, \]
which using integration by parts can be written as
\[ dH_{(z,q,p)}(\hat{z}, \hat{q}, \hat{p}) = \kappa \hat{z} + \int_\R -\frac{\dee^2 q}{\dee x^2} \, \hat{q} + p\,\hat{p} \ dx \]
since the derivatives must tend to 0 at infinity in both directions.

Then, using the expression $\flat(X_H) = dH -(R(H)+H)\eta$ and writing 
\[X_H = (\hat{z}_H, \hat{q}_H, \hat{p}_H)\]
we obtain two expressions for $\flat(X_H)(Y)$ where $Y=(\tilde{z}, \,\tilde{q},\, \tilde{p})$:

\begin{align*}
    \flat(X_H)(Y) &= \eta(X_H)\tilde{z} + \int_\R \left( -\hat{p}_H - p\cdot\eta(X_H) \right) \tilde{q} + \int_\R \hat{q}_H\cdot \tilde{p} \\
    &= -H\cdot\tilde{z} + \int_\R \left( -\frac{\dee^2 q}{\dee x^2}+(\kappa + H)p \right)\tilde{q} + \int_\R p \cdot \tilde{p}.
\end{align*}
For these to be equal we simply require that each term is equal for all $\tilde{z}$, $\tilde{q}$ and $\tilde{p}$. By the Riesz representation theorem, this is equivalent to the following system of equations:
\[\begin{cases}
    \eta(X_H) = -H \\
    \hat{q}_H = p \\
    \hat{p}_H = \frac{\dee^2 q}{\dee x^2} - \kappa p.
\end{cases}\]
An integral curve of $X_H$ must satsify $(\dot{z}, \dot{q}, \dot{p}) = (\hat{z}_H, \hat{q}_H, \hat{p}_H)$ and so combining this with the above, we find that $q(t)$ must satisfy
\[ \frac{\dee^2 q}{\dee t^2} = \frac{\dee^2 q}{\dee x^2} - \kappa \frac{\dee q}{\dee t}, \]
that is, it must be a solution to the damped wave equation.

\section{Cosymplectic and cocontact structures}
\subsection{Cosymplectic manifolds}
Theorem \ref{thm: inf dim contact forms} is in fact a special case of the following. The terminology  is borrowed from \cite{Albert89} (\emph{structure de presque contact}).
\begin{theorem}[Almost contact structure]
    Let $M$ be a manifold modelled on some lcTVS, $E$. Let $\theta \in \Omega^1(M)$ and $\omega\in \Omega^2(M)$. Assume that $\omega$ is degenerate. Define $\Hor : = \ker \theta$ and $\Ver := \ker \omega$. TFAE:
    \begin{enumerate}
        \item $TM = \Hor \oplus \Ver$.
        \item $\omega|_\Hor$ is weakly non-degenerate.
        \item The map \[\flat_{\theta,\omega}\colon \vect(M)\rightarrow \Omega^1(M)\]
        \[ \flat_{\theta,\omega}(X) = \iota_X \omega + \theta(X)\theta\]
        is an injective map (homomorphism of $\Cinf(M)$-modules).
    \end{enumerate}
\end{theorem}
\begin{proof}
    Exactly the same as in theorem \ref{thm: inf dim contact forms}, but replacing $d\eta$ with $\omega$ and $\eta$ with $\theta$.
\end{proof}

\begin{definition}[Almost contact structure]
    If $\theta \in \Omega^1(M)$ and $\omega \in \Omega^2(M)$ satisfy any (and hence all) of the above conditions then we call the pair $(\theta, \omega)$ an \emph{almost contact structure} on $M$ and $(M, \theta, \omega)$ an \emph{almost contact manifold.}
\end{definition}

The above generalises the property that, for a 1-form $\theta$, and a 2-form $\omega$ on a $2n+1$-dimensional manifold $M$, $\theta \wedge \omega^n > 0$.
In particular, when $\theta = \eta$ and $\omega = d\eta$ for some $\eta \in \Omega^1(M)$, then we recover theorem \ref{thm: infinite-dimensional analogue of definition} and $(M, \eta, d\eta)$ is a contact manifold.

Otherwise, when $\theta$ and $\omega$ are both closed, then we call $(M, \theta, \omega)$ a \emph{cosymplectic manifold}, in line with the terminology used in finite dimensions. Refer to \cite[\textsection 5]{deLeonIzquierdo24}, for example. 

\begin{definition}
    The \emph{Reeb field} is defined similarly to before. It is the unique vector field $R \in \vect(M)$ for which $\iota_R \theta = 1$ and $\iota_R \omega = 0$.
\end{definition}

\begin{definition}
    Let $(M,\theta,\omega)$ be a cosymplectic manifold, and $H\in \Cinf(M)$. Then the \emph{evolution vector field} associated to $H$ is defined by the following expression:
    \[\flat(\E_H) = dH + (1-R(H))\theta.\]
    We call $(M, \theta, \omega, H)$ a Hamiltonian system. Any integral curve of $\E_H$ describes the evolution of the system under the Hamiltonian $H$ given some initial condition $p \in M$.
\end{definition}

In finite dimensions, cosymplectic manifolds can be used to model time-dependent Hamiltonian mechanics, as in the following example. In the context of such systems, the Reeb field can be thought of as defining the flow of time, in other words, it tells the ``arrow of time'' where to fly. With the following example we show that this extends to infinite dimensions as well.

\subsubsection{The wave equation with a source term}
\begin{example}
    Define $P:= W_2^2(\R)$ and $Q = W_2^2(\R)' \cong W_2^2(\R)$. Consider $M$ to be the Banach space $\R \times Q \times P$ which we will treat as a Banach manifold modelled by itself. We will write a point in $M$ as $(t, q, p)$; physically $t\in \R$ represents time, and $q$ and $p$ are functions on a 1-dimensional string representing respectively the height of each point and the momentum of an infinitesimal mass located at that point.

    Let $\theta = dt \in \Omega^1(M)$ and $\omega \in \Omega^2(M)$ be the canonical symplectic structure when restricted to $T'W^2_2(\R) \cong Q\times P$, and then parallelised with respect to the vector space structure of $M$:
    \[ \omega_{(t,q,p)} ((\hat{t}, \hat{q}, \hat{p}),\, (\tld{t}, \tld{q}, \tld{p})) = \int_\R \tld{p}\hat{q} - \hat{p}\tld{q}.\]
    Alternatively we can define $\omega$ as the pullback of the canonical symplectic form $\omega_0 \in \Omega^2(P \times Q)$ via the projection $\pi \colon \R \times P \times Q \rightarrow P\times Q$.

    With $\theta$ and $\omega$ defined as above, the map
    \[ \flat_{\theta,\omega}\colon X \mapsto \iota_X \omega + \theta(X)\theta \]
    is an injective homomorphism $\vect(M) \rightarrow \Omega^1(M)$. The Reeb field in this case is $R = \frac{\dee}{\dee t}$. 

    Now set the Hamiltonian to be
    \[ H = U(t,q,p) + \int_\R \frac{1}{2}\left( \frac{\dee q}{\dee x} \right)^2 + \frac{1}{2} p^2 \, dx \]
    for some \emph{source potential} $U \in \Cinf(\R \times Q \times P)$. Note the similarity to the Hamiltonian for the ordinary wave equation, apart from the added term.

    We calculate 
    \[ dH_{(t,q,p)}(\tld{t}, \tld{q}, \tld{p}) = dU_{(t,q,p)}(\tld{t}, \tld{q}, \tld{p}) + \int_\R - \frac{\dee^2 q }{\dee x ^2} \tld{q} + p \, \tld{p} \ dx. \]

    We briefly pay special attention to $dU$. Since $dU_{(t,q,p)}$ is a linear function $\R \times Q \times P \rightarrow \R$ we may write it as the sum of restrictions onto $\R$, $Q$ and $P$:
    \[ dU_{(t,q,p)} = dU^t + dU^Q + dU^P\]
    or 
    \[dU_{(t,q,p)}(\tld{t}, \tld{q}, \tld{p}) = \tld{t} \cdot \left. \frac{\dee U}{\dee t} \right|_{q,p} + dU^Q_{(t,q,p)}(\tld{q}) + dU^P_{(t,q,p)}(\tld{p}).\]
    For now, we ignore the $dU^t$ term (it is about to be cancelled out by another term) but $dU^Q_{(t,q)}$ is a bounded linear map $Q \rightarrow \R$ and so by the Riesz representation theorem we can write this as
    \[ dU^Q_{(t,q,p)}(\tld{q}) = \int_\R f^U_{(t,q,p)}\cdot \tld{q} \ dx \]
    for some map
    \[f_{(-)}^U \colon (t,q,p) \mapsto f_{(t,q,p)}^U \in W^2_2(\R).\]
    Similarly we write 
    \[ dU^P_{(t,q,p)}(\tld{p}) = \int_\R g^U_{(t,q,p)}\cdot \tld{p} \ dx\]
    for some $g^U_{(-)}$.

    Hence (now supressing coordinates)
    \[dH(\tld{t}, \tld{q}, \tld{p}) = \frac{\dee U}{\dee t} dt(\tld{t}) + \int_\R \left(f_{ (t,q,p)}^U -\frac{\dee ^2 q}{\dee x^2}   \right) \tld{q} + \left( g^U_{(t,q,p)} + p \right) \, \tld{p} \ dx \]

    Now, for $\flat(\E_H)$, calculate
    \begin{align*}
        \flat(\E_H) & = dH + (1-R(H))\theta \\
        &= dH - \frac{\dee H}{\dee t} dt + dt, \\
    \end{align*}
    and noting that $\frac{\dee H}{\dee t} = \frac{\dee U}{\dee t}$,
    \begin{align*}
        \flat(\E_H) (\tld{t}, \tld{q}, \tld{p}) &= \cancel{\frac{\dee U}{\dee t} dt(\tld{t})} + \int_\R \left(f_{ (t,q,p)}^U -\frac{\dee ^2 q}{\dee x^2}   \right) \tld{q} + \left( g^U_{(t,q,p)} + p \right) \, \tld{p} \ dx - \cancel{\frac{\dee U}{\dee t} dt(\tld{t})} + dt \\
        &= \int_\R \left(f_{ (t,q)}^U -\frac{\dee ^2 q}{\dee x^2}   \right) \tld{q} + \left( g^U_{(t,q,p)} + p \right) \, \tld{p} \ dx + \tld{t}
    \end{align*}
    
    Write $\E_H = (\hat{t}_H, \hat{q}_H, \hat{p}_H)$, and then to solve for $\E_H$ we ``compare coefficients'' between the above expression and 
    \[\flat(\E_H)(\tld{t}, \tld{q}, \tld{p}) = \int_\R -\hat{p}_H \tld{q} + \hat{q}_H \tld{p} \ dx + \hat{t}_H \tld{t} \]
    to obtain
    \begin{align*}
        \hat{q}_H &= p + g^U_{(t,q,p)}  , & \hat{p}_H &= \frac{\dee ^2 q}{\dee x^2} -f_{(t,q,p)}^U, & \hat{t}_H &= 1.
    \end{align*}
    Solving for the flow ($\dot{t} = \hat{t}_H$, $\dot{q} = \hat{q}_H$, $\dot{p} = \hat{p}_H$) gives 
    \[ \frac{\dee^2 q}{\dee t^2} = \frac{\dee^2 q}{dx^2} - f^U_{(t,q,p)}+ \frac{\dee}{\dee t} g^U_{(t,q,p)}  \]
    which describes a wave equation with a source. 

    Note that as $f_{(-)}^U$ and $g^U_{(-)}$ depend globally on the extended phase space, this is in fact more general than the normal source term. One obvious case is when $U$ is given by the following, for some $u(t) = u(x,t)$:
    \[U(t,q) = \frac{1}{2}\int_\R (q-u(t))^2 \ dx.\]
    This is independent of $p$, so $g^U_{(-)} \equiv 0$ and
    \[ dU^Q_{(q,t)}(\hat{q}) = \int_\R 2 (q-u(t)) \cdot \hat{q} \ dx. \]
    Hence $f^U_{(t,q)} = (q-u(t))$ and 
    \[ \frac{\dee^2 q}{\dee t^2} = \frac{\dee^2 q}{dx^2} - q + u(t).\]

    The more familiar source term can be obtained by setting 
    \[U(q,t) = \int_\R q \cdot u(t) \ dx \]
    in which case 
    \[dU^Q_{(q,t)}(\hat{q}) = \int_\R u(t) \cdot \hat{q} \quad \Rightarrow \quad f^U_{(t,q)} = u(t) \]
    and 
    \[ \frac{\dee^2 q}{\dee t^2} = \frac{\dee^2 q}{dx^2} + u(t).\]
\end{example}

\subsubsection{A cosymplectic Darboux theorem}
Let $(M,\theta, \omega)$ be a cosymplectic manifold and assume that $R$ has a unique flow, $\Fl^t$. Clearly $\omega$ is invariant under $\Fl^t$ because 
\[\Lie_R \omega = \iota_R d\omega + d \iota_R \omega = 0.\]
By Theorem 2 of \cite{Teichmann01}, since our vertical distribution is of rank 1, and hence involutive, we can write a Frobenius chart in which $R$ is constant. Then the existence of Darboux charts on open subsets of the whole manifold can be reduced to the question of whether the Darboux theorem holds on any leaf of the above foliation.

\subsection{Cocontact manifolds}
In finite dimensions, cocontact manifolds combine the properties of contact and cosymplectic manifolds. That is, we can use them to model both time-dependence and dissipation simultaneously. Here we must generalise statements of the form
\[\theta \wedge \eta \wedge (d\eta)^n > 0\]
from $2n+2$-dimensional manifolds to infinite-dimensional manifolds.

\begin{theorem}[Cocontact structure]
    Let $M$ be a manifold and $\theta, \eta$ be 1-forms on $M$ such that $\theta$ is closed. Define $\Hor_t:= \ker \theta$, $\Hor_z = \ker \eta$ and $\Hor_{tz} = \Hor_t \cap \Hor_z$. Further define $\Ver := \ker d \eta$ and then $\Ver_t:= \Ver \cap \Hor_z$ and $\Ver_z:= \Ver \cap \Hor_t$. We will also assume that $d\eta|_{\Hor_t}$ and $d\eta_{\Hor_z}$ are degenerate on their respective domains. TFAE:
    \begin{enumerate}
        \item $TM = \Hor_{tz} \oplus \Ver_t \oplus \Ver_z$. 
        \item $(\Hor_{tz}, d\eta|_{\Hor_{tz}})$ is a weakly symplectic distribution.
        \item The map $\flat_{\theta,\eta}\colon v \mapsto  \theta(v)\theta + \iota_v d\eta + \eta(v)\eta$ is an injective bundle morphism $TM \rightarrow T'M$.
    \end{enumerate}
\end{theorem}

\begin{proof}
    We prove $(1) \Rightarrow (2) \Rightarrow (3) \Rightarrow (1)$.

    $(1)\Rightarrow(2)$: Assume that $TM = \Hor_{tz} \oplus \Ver_t \oplus \Ver_z$. Then 
    \begin{align*}
        \braces{0} & = \Hor_{tz} \cap \Ver_t \cap \Ver_z \\
        & = \Hor_{tz} \cap (\Ver \cap \Hor_z) \cap (\Ver \cap \Hor_t) \\
        & = \Hor_{tz} \cap \Ver \\
        & = \Hor_{tz} \cap \ker d\eta \\
        & = \ker d\eta|_{\Hor_{tz}}
    \end{align*}
    and hence $(\Hor_{tz}, d\eta|_{\Hor_{tz}})$ is a weakly symplectic distribution.

    $(2)\Rightarrow(3)$: Assume that $d\eta|_{\Hor_{tz}}$ is weakly non-degenerate, and let $X \in \vect(M)$. Assume that $\flat(X) = 0 $. Then 
    \[\flat(X)(X) = \theta(X)\cdot \theta(X) + d\eta(X,X) + \eta(X) \cdot \eta(X) = 0\]
    \[\Rightarrow \theta(X)^2 + \eta(X)^2 = 0 \quad \Rightarrow \quad \theta(X) = \eta(X) = 0,\]
    so $X$ is a section of $\Hor_{tz}$, and $\flat(X) = \iota_X d\eta = 0$. Since $d\eta$ is non-degenerate when restricted to $\Hor_{tz}$, we conclude that $X=0$ and hence $\flat$ is injective.

    $(3)\Rightarrow(1)$: Assume that $\flat_{\theta, \eta}$ is injective. To show that the pairwise intersections of $\Hor_{tz}$, $\Ver_t$, and $\Ver_z$ are trivial, note that
    \[ \Hor_{tz} \cap \Ver_t = \Hor_{tz} \cap \Ver_z = \Ver_t \cap \Ver_z  = \Hor_t \cap \Hor_z \cap \Ver. \]
    Now fix $p\in M$ and let $v \in \Hor_t \cap \Hor_z \cap \Ver \subset T_pM$. Then
    \[\flat(v) = \theta(v)\theta + \iota_v d\eta + \eta(v)\eta = 0\]
    and $v=0$ by injectivity of $\flat$. Therefore,
    \[ \Hor_{tz} \cap \Ver_t = \Hor_{tz} \cap \Ver_z = \Ver_t \cap \Ver_z  = \Hor_t \cap \Hor_z \cap \Ver = \braces{0}. \]

    Now, because $\Hor_{tz} = \ker \theta|_{\Hor_z}$ and $\Hor_z$ is itself defined as the kernel of a 1-form, $\Hor_{tz}$ must have codimension either 1 or 2 wrt $T_pM$. Also the spaces $\Ver_t$ and $\Ver_z$ are each non-trivial (by the assumption that $d\eta$ is degenerate on both $\Hor_t$ and $\Hor_z$) and have trivial intersection, hence the space $\Ver_t \oplus \Ver_z$ has dimension at least 2. Finally, it follows from the codimension lemma (lemma \ref{codimension lemma}) that $TM = \Hor_{tz} \oplus \Ver_t \oplus \Ver_z$.
\end{proof}

\begin{definition}
    If the tuple $(M,\theta, \eta)$ satisfies the assumptions above as well as any (hence all) of the conditions (1)-(3), then we call $(M,\theta, \eta)$ a \emph{cocontact manifold}.
\end{definition}

\begin{remark}
    It is necessary to include the assumption that $d\eta$ is non-degenerate on both $\Hor_t$ and $\Hor_z$ seperately, as otherwise we would be allowing situations where, for example, $\theta = 0$ and $(M,\eta)$ is a contact manifold in its own right. In the finite-dimensional case it is unnecessary because the finite dimension forces both $\Ver_t$ and $\Ver_z$ to be non-trivial.
\end{remark}

Given a cocontact manifold $(M,\theta, \eta)$ we define two vector fields, the $t$-Reeb field given by
\[ \flat_{\theta, \eta}(R_t) = \theta \]
and the $z$-Reeb field given by 
\[ \flat_{\theta, \eta}(R_z) = \eta \]
whenever they exist (there is, to the author's knowledge, no \textit{a priori} guarantee that $\theta$ and $\eta$ lie in the image of $\flat$).

Given $H\in \Cinf(M)$, the \emph{Hamiltonian vector field} for a cocontact Hamiltonian system $(M, \theta, \eta, H)$ is given by
\[ \flat_{\theta, \eta}(X_H) = dH - (R_z(H)+H)\eta + (1-R_t(H))\theta \]

For the sake of brevity the following example is left as an exercise to the reader.
\begin{example}[Damped wave equation with source term]
    Let $M =\R^2 \times W_2^2(\R) \times W_2^2(\R)$ with typical element $(t,z,q,p)$, paired with the cocontact structure
    \[\theta = dt, \quad \eta(\hat{t},\zhat, \hat{q}, \hat{p}) = \hat{z} - \int_\R p \, \hat{q} \ dx, \]
    and with Hamiltonian function
    \[H = \kappa z + \int_\R q \cdot u(t) \ dx + \int_\R \frac{1}{2}\left( \frac{\dee q}{\dee x} \right)^2 + \frac{1}{2} p^2 \, dx,\]
    for some $u(t) \in \Cinf(\R, W^2_2)$.
    
    Then the flows of $X_H$ give solutions to the damped wave equation with a source:
    \[ \frac{\dee^2 q}{\dee t^2} = \frac{\dee^2 q}{\dee x^2} - \kappa \frac{\dee q}{\dee t} + u(t).\]
\end{example}
\appendix
\section{Useful Lemmata}
\begin{definition}[Codimension]
    Let $W$ be a vector space, and $V\subseteq W$ a subspace. Then the codimension of $V$ with respect to $W$ is the dimension of the quotient $W/V$. Equivalently, if $U$ is any compliment to $V$ (so that $W=U\oplus V$), then the codimension of $V$ is the dimension of $U$. We write
    \[\codim_W V := \dim (W/V) = \dim U. \]
\end{definition}

\begin{lemma}[Codimension lemma]\label{codimension lemma}
    Let $W$ be an infinite-dimensional real vector space, and let $U,\ V$ be subspaces of $W$ such that $U\cap V = \braces{0}$. Then $\dim U \leq \codim_W V$.
    
    In particular, if $\dim U = \codim_W V$ is finite, then $W = U\oplus V$.
\end{lemma}
\begin{proof}
    We have the following short exact sequence of vector spaces:
    \[
    \begin{tikzcd}
        0 \arrow[r] & V \arrow[r, "i", hook] & W \arrow[r, "q", two heads] & W/V \arrow[r] & 0
    \end{tikzcd}
    \]
    where $i$ is an inclusion, and $q$ is the projection onto the quotient spaces. Consider the restriction of $q$ to $U$; this is surjective onto its image, $U/V$, and it is injective because for $u,\ u' \in U$,
    \[ q(u) = q(u') \ \Rightarrow \ u-u' \in V\cap U \ \Rightarrow \ u=u'. \]
    So $q|_U \colon U \rightarrow U/V \subsetneq W/V$ is an isomorphism of $U$ onto a strict subspace of $W/V$. We hence conclude that 
    \[ \dim U = \dim U/V \leq \dim W/V = \codim_W V \]

    For the second part of the proof, further assume that $\dim U = \codim_W V = k \in \mathbb{N}$. Then $\dim U/V = \dim W/V$, i.e. $U/V = W/V$. By assumption, $U\cap V = \braces{0}$, so it remains to show that $W = U+V$. So pick some $w\in W$. Since $q$ is an isomorphism onto $U/V = W/V$, there exists some $u\in U$ such that $q(u) = q(w)$, i.e. such that $u+V = w+V$. Hence 
    \[u-w = v  \ \Rightarrow \  w = u-v\]
    for some $v \in V$, which completes the proof.
\end{proof}
\begin{definition}
    Given any \emph{abelian} category, a short exact sequence is a diagram of the form
    \[
    \begin{tikzcd}
        0 \arrow[r] & A \arrow[r, "i", hook] & B \arrow[r, "\lambda"] & C \arrow[r] & 0,
    \end{tikzcd}
    \]
    where $A = \ker \lambda$, $C = \im \lambda$, and $i$ is the inclusion $A \hookrightarrow B$.
    It is said to \emph{split} if there exist maps $f\colon B \rightarrow \ker \lambda$ and $g\colon \im \lambda \rightarrow B$ such that $f\circ i = \id_{\ker \lambda}$ and $\lambda \circ g  = \id_{\im \lambda}$.
\end{definition}

Note that the category of vector spaces is abelian (and in fact every short exact sequence splits), but the category of topological vector spaces is not. However, the following lemma is a much weaker statement, applying when $\lambda\colon B\rightarrow \R$.

\begin{lemma}
    Let $E$ be a Fréchet space, $\lambda\colon E \rightarrow \R$ be a continuous linear map and $F = \ker \lambda$. Then the following short exact sequence splits inside the category of topological vector spaces and continuous linear maps:
    \[
    \begin{tikzcd}
        0 \arrow[r] & F \arrow[r, "i", hook] & E \arrow[r, "\lambda"] & \R \arrow[r] & 0
    \end{tikzcd}
    \]
    and $E \cong F \oplus \R$.
\end{lemma}
\begin{proof}
    Let $v \in E$ such that $\lambda(v) = 1$, which exists by the Hahn-Banach theorem, and let $P$ be the projection along $v$ onto $\ker f$ given by $P(x) = x - \lambda(x)v$. Then clearly $P\circ i = \id_F$. Then, consider the map $-\cdot v \colon R \rightarrow E$ which has image $\angles{v}$. Clearly $\lambda(r\cdot v) = r$ so along with the previous statement, and the fact that all maps so far are continuous by definition of topological vector space, we conclude that the above short exact sequence is split. 

    Now, to show that $E\cong F \oplus \R$, consider the maps $f\colon F\oplus \R \rightarrow E$ and $g \colon F \oplus \R \rightarrow E$ given by
    \begin{align*}
        f\colon & (x,r) \mapsto x+r\cdot v \\
        g\colon & y \mapsto (P(y), \lambda(y))
    \end{align*}
    which are both continuous and linear, and it is very easy to check that $f\circ g = \id_{F\oplus \R}$ and $g \circ f = \id_E$.
\end{proof}

\subsubsection{Direct sums of symplectic vector spaces}

\begin{lemma}
    $(V\oplus W)' \cong V' \oplus W'$
\end{lemma}

\begin{proof}
    Let $f: V\rightarrow \R$ and $g\colon W \rightarrow \R$ be toplinear maps, and consider the map $F \colon (f,g) \mapsto \varphi_{(f,g)}$ where
    \[\varphi_{(f,g)}(v,w) = f(v) + g(w).\]
    Clearly, $\varphi_{(f,g)}\colon V\oplus W\rightarrow \R $ is linear and continuous, so $F\colon V'\oplus W' \rightarrow (V\oplus W)'$.

    Now let $\varphi \in (V\oplus W)'$ and consider $G\colon \varphi \mapsto (\varphi\circ i_V, \varphi\circ i_W)$, where $i_V\colon V\rightarrow \oplus W$, $i_W\colon W\rightarrow V\oplus W $ are the obvious inclusions. Since the inclusions and $\varphi$ are continuous, $G\colon (V\oplus W)'\rightarrow V'\oplus W'$.

    Finally, since $F\circ G= \id_{(V\oplus W)'}$ and $G\circ F = \id_{V'\oplus W'}$, we have $(V\oplus W)' \cong V' \oplus W'$.
\end{proof}

\begin{lemma}\label{lem: weakly symplectic direct sums}
    Let $(V, \Omega_1)$ and $(W, \Omega_2)$ be weakly symplectic vector spaces. Define $\Omega \in \bigwedge^2 (V\oplus W)'$ by \[\Omega(v+w, v'+w') = \Omega_1(v, v') + \Omega_2(w,w')\] for $v,v' \in V$, $w,w' \in W$. Then $(V\oplus W, \Omega)$ is weakly symplectic.
\end{lemma}

\begin{proof}
    Let $v_1 + w_1 \in V\oplus W$ be such that for all $v_2+ w_2 \in V \oplus W$, $\Omega(v_1+w_1, v_2 + w_2)=0$. Then,
    \begin{align*}
        &\forall v_2 \in V, \ \forall w_2 \in W,& \ 0 &= \Omega(v_1+w_1, v_2 + w_2) \\
        &&&= \Omega_1(v_1, v_2) + \Omega_2(w_1, w_2) \\
        &\Rightarrow \forall v_2 \in V, \ \forall w_2 \in W, & -\Omega_1(v_1,v_2) &= \Omega(w_1, w_2) 
    \end{align*}
    which, by weak non-degeneracy of $\Omega_1$ and $\Omega_2$, implies that for all $v_2$, $\Omega(v_1, v_2)=0$ and hence that $v_1 = 0$. Similarly we conclude that $w_1=0$, and so $v_1+w_1 = 0$ and that $\Omega$ is weakly non-degenerate.
\end{proof}

\begin{lemma}\label{lem: strongly symplectic direct sums}
    Let $(V,\Omega_1)$, $(W,\Omega_2)$ be strongly symplectic. Then $(V\oplus W, \Omega)$ as defined above is strongly symplectic.
\end{lemma}

\begin{proof}
    Define $\tilde{\Omega}_1\colon v \mapsto \Omega_1(v,-)$ and similarly define $\tilde{\Omega}_2$, as maps from $V, W$ resp. into their continuous duals $V', W'$. Finally define $\tilde{\Omega}\colon (V\oplus W) \rightarrow (V \oplus W)' \cong V'\oplus W'$,
    \[\tilde{\Omega}\colon v+w \mapsto \Omega(v+w,-)\cong \Omega_1(v,-) + \Omega_2(w,-).\]

    $\Omega_1$ and $\Omega_2$ are strongly non-degenerate so the maps $\tilde{\Omega}_1$ and $\tilde{\Omega}_2$ have inverses $\tilde{\Omega}_1^{-1}\colon V' \rightarrow V$ and $\tilde{\Omega}_2^{-1}\colon W' \rightarrow W$.

    We can check with a few lines of calculation that $\tilde{\Omega}$ has inverse
    \[\tilde{\Omega}^{-1}\colon (V\oplus W)' \cong V'\oplus W' \rightarrow V\oplus W\]
    given by 
    \[\tilde{\Omega}^{-1}(\alpha) = \tilde{\Omega}_1^{-1}(\alpha|_V)+\tilde{\Omega}_2^{-1}(\alpha|_W).\]
    Simply check that $\tilde{\Omega}\circ\tilde{\Omega}^{-1} = \id_{(V\oplus W)}$ and $\tilde{\Omega}^{-1}\circ\tilde{\Omega} = \id_{(V\oplus W)'}$.
\end{proof}
\printbibliography

\end{document}